\pdfoutput=1
\documentclass[11pt]{article}

\usepackage{enumerate}
\usepackage{fancyhdr,amsmath,amssymb,amsthm,url}

\usepackage{booktabs}
\usepackage{float}
\usepackage{multirow}
\usepackage[normalem]{ulem}
\usepackage{accents}
\usepackage{graphicx}				%
\usepackage{amssymb}
\usepackage{color}
\usepackage{xspace}
\usepackage{array}

\usepackage{thmtools}
\usepackage{thm-restate}

\usepackage{algorithm}

\usepackage{etoolbox}

\newtoggle{arxiv}
\newcommand{\arxiv}[1]{\iftoggle{arxiv}{#1}{}}

\toggletrue{arxiv}

\arxiv{
	\usepackage[top=1in, right=1in, left=1in, bottom=1in]{geometry}
	\usepackage[hidelinks]{hyperref}
	\def\addcontentsline#1#2#3{}
	\usepackage[round]{natbib}
	
	\renewcommand{\cite}[1]{\citep{#1}}
	
}

\makeatletter
\long\def\@makecaption#1#2{
	\vskip 0.8ex
	\setbox\@tempboxa\hbox{\small {\bf #1:} #2}
	\parindent 1.5em  
	\dimen0=\hsize
	\advance\dimen0 by -3em
	\ifdim \wd\@tempboxa >\dimen0
	\hbox to \hsize{
		\parindent 0em
		\hfil 
		\parbox{\dimen0}{\def\baselinestretch{0.96}\small
			{\bf #1.} #2
		} 
		\hfil}
	\else \hbox to \hsize{\hfil \box\@tempboxa \hfil}
	\fi
}
\makeatother

\usepackage[noend]{algpseudocode}

\input{macros.sty} %

\arxiv{
  \title{``Convex Until Proven Guilty'': 
    Dimension-Free Acceleration of Gradient 
    Descent on Non-Convex Functions}

  \author{Yair Carmon ~~~ John C.\ Duchi ~~~ Oliver Hinder ~~~ Aaron Sidford \\
    \texttt{\{\href{mailto:yairc@stanford.edu}{yairc},\href{mailto:jduchi@stanford.edu}{jduchi},\href{mailto:ohinder@stanford.edu}{ohinder},\href{mailto:sidford@stanford.edu}{sidford}\}@stanford.edu}}
  \date{}
}

\begin{document}

\arxiv{\maketitle}

\begin{abstract}
We develop and analyze a variant of Nesterov's accelerated gradient descent (AGD) for minimization of smooth non-convex functions. We prove that one of two cases occurs: either our AGD variant converges quickly, as if the function was convex, or we produce a certificate that the function is ``guilty'' of being non-convex. This non-convexity certificate allows us to exploit negative curvature and obtain deterministic, dimension-free acceleration of convergence for non-convex functions. For a function $f$ with Lipschitz continuous gradient and Hessian, we compute a point $x$ with $\|\nabla f(x)\| \le \epsilon$ in $O(\epsilon^{-7/4} \log(1/ \epsilon) )$ gradient and function evaluations.  Assuming additionally that the third derivative is Lipschitz, we require only $O(\epsilon^{-5/3} \log(1/ \epsilon) )$ evaluations.
\end{abstract}
\section{Introduction}

Nesterov's seminal \citeyear{nesterov1983method} accelerated
gradient method
has inspired substantial
development of first-order methods for large-scale
convex 
optimization. In recent years,
machine learning and statistics have seen a shift toward large scale
\emph{non-convex} problems, including methods for matrix
completion~\cite{koren2009matrix}, phase retrieval~\cite{CandesLiSo15,
WangGiEl16},
dictionary learning~\cite{MairalBaPoSaZi08}, and neural network
training~\cite{lecun2015deep}. In practice, techniques from
accelerated gradient methods---namely, momentum---can have substantial
benefits for stochastic gradient methods, for example, in training neural
networks~\cite{RumelhartHiWi86, kingma2014adam}.  Yet little of the rich
theory of acceleration for convex optimization is known to transfer into non-convex
optimization.

Optimization becomes more difficult without convexity, as
gradients no longer provide global information about the function.  Even
determining if a stationary point is a local minimum is (generally) NP-hard
\cite{MurtyKa87,Nesterov00}.  It is, however, possible to leverage
non-convexity to improve objectives in smooth optimization: moving in 
directions of
negative curvature can guarantee function value reduction. We explore the
interplay between negative curvature, smoothness, and acceleration
techniques, showing how an understanding of the three simultaneously yields
a method that provably accelerates convergence of gradient descent for a 
broad class of non-convex functions.

\subsection{Problem setting}  
We consider the unconstrained minimization problem
\begin{equation}
  \label{eqn:obj}
  \minimize_x f(x),
\end{equation}
where $\func$ is smooth but potentially non-convex. We assume throughout the
paper that $f$ is bounded from below, two-times differentiable, and has
Lipschitz continuous gradient and Hessian. In Section~\ref{sec:third-order}
we strengthen our results under the additional assumption that $f$ has
Lipschitz continuous third derivatives. Following the standard first-order 
oracle model \cite{NemirovskiYu83}, we consider optimization methods that 
access only values and gradients of $f$ (and not higher order derivatives), 
and we measure their complexity by the total number of gradient and function 
evaluations.

Approximating the global minimum of $f$ to
$\epsilon$-accuracy is generally intractable, requiring time exponential in
$d\log\frac{1}{\epsilon}$ \citep[\S1.6]{NemirovskiYu83}.  Instead, we seek
a point $x$ that is $\epsilon$-approximately stationary, that is,
\begin{equation}
  \label{eq:approx-stationary-objective}
  \norm{\grad f(x)} \le \epsilon.
\end{equation}
Finding stationary points is a canonical problem in nonlinear optimization
\cite{NocedalWr06}, and while saddle points and local maxima are
stationary, excepting pathological cases, descent methods that
converge to a stationary point converge to a local
minimum~\citep[\S3.2.2]{lee2016gradient,Nemirovski1999}.

If we assume $f$ is convex, gradient descent satisfies the bound~\eqref{eq:approx-stationary-objective} after $O(\epsilon^{-1})$ 
gradient evaluations,
and AGD improves this rate to $O(\epsilon^{-1/2}\log\frac{1}{\epsilon})$
\cite{NesterovGradSmall2012}. Without convexity, gradient descent is significantly worse, having
worst-case complexity $\Theta(\epsilon^{-2})$ \cite{cartis2010complexity}.
More sophisticated gradient-based methods, including nonlinear conjugate
gradient~\cite{hager2006survey} and L-BFGS~\cite{liu1989limited} provide
excellent practical performance, but their global convergence guarantees are 
no better than $O(\epsilon^{-2})$.
Our work~\cite{carmon2016accelerated} and, independently,
\citet{agarwal2016finding},
break this $O(\epsilon^{-2})$ barrier, obtaining the rate  
$O(\epsilon^{-7/4}\log \frac{d}{\epsilon})$. Before we discuss this
line of work in 
Section~\ref{sec:related-work}, we overview our contributions.

\subsection{Our contributions}

\paragraph{``Convex until proven guilty''}

Underpinning our results is the observation that when we run Nesterov's
accelerated gradient descent (AGD) on \emph{any} smooth function $f$, one of
two outcomes must follow:

\begin{enumerate}[(a)]
\item \label{item:agd-good}
  AGD behaves as though $f$ was $\StrConv$-strongly convex, 
  satisfying inequality~\eqref{eq:approx-stationary-objective} in 
  $O(\StrConv^{-1/2}\log\frac{1}{\epsilon})$ iterations.
\item \label{item:agd-guilty}
  There exist points $u,v$ in the AGD trajectory that prove $f$ is 
  ``guilty'' of 
  not being $\StrConv$-strongly convex,
  \begin{equation}\label{eq:proven-guilty}
    f(u) < f(v) + \grad f(v) ^T (u-v) + \frac{\StrConv}{2}\norm{u-v}^2.
  \end{equation}
\end{enumerate}
\noindent
The intuition behind these observations is that if
inequality~\eqref{eq:proven-guilty} never holds during the iterations of
AGD, then $f$ ``looks'' strongly convex, and the
convergence~\eqref{item:agd-good} follows.  In
Section~\ref{sec:alg-components} we make this observation precise,
presenting an algorithm to monitor AGD and quickly find the witness pair
$u,v$ satisfying \eqref{eq:proven-guilty} whenever AGD progresses more
slowly than it does on strongly convex functions.  We believe there is
potential to apply this strategy beyond AGD, extending additional convex
gradient methods to non-convex settings.

\paragraph{An accelerated non-convex gradient method}

In Section~\ref{sec:main-result} we propose a method that iteratively
applies our monitored AGD algorithm to $f$ augmented by a proximal
regularizer. We show that both outcomes~\eqref{item:agd-good}
and~\eqref{item:agd-guilty} above imply progress minimizing $f$, where in
case~\eqref{item:agd-guilty} we make explicit use of the negative curvature
that AGD exposes. These progress guarantees translate to an overall
first-order oracle complexity of $O(\epsilon^{-7/4}\log\frac{1}{\epsilon})$,
a strict improvement over the $O(\epsilon^{-2})$ rate of gradient
descent. In Section~\ref{sec:experiments} we report preliminary experimental
results, showing a basic implementation of our method outperforms gradient
descent but not nonlinear conjugate gradient.

\paragraph{Improved guarantees with third-order smoothness}

As we show in Section~\ref{sec:third-order}, assuming Lipschitz continuous
third derivatives instead of Lipschitz continuous Hessian allows us to
increase the step size we take when exploiting negative curvature, making
more function progress. Consequently, the complexity of our method improves
to $O(\epsilon^{-5/3}\log\frac{1}{\epsilon})$. While the analysis of the
third-order setting is more complex, the method remains essentially
unchanged.  In particular, we still use only first-order information, never
computing higher-order derivatives.

\subsection{Related work}\label{sec:related-work}

\citet{nesterov2006cubic} show that cubic regularization of Newton's
method finds a point that satisfies the stationarity
condition~\eqref{eq:approx-stationary-objective}
in $O(\epsilon^{-3/2})$ evaluations of the Hessian. Given sufficiently
accurate arithmetic operations, a Lipschitz continuous Hessian is approximable to
arbitrary precision using finite gradient differences, and obtaining a full
Hessian requires $O(d)$ gradient evaluations. A direct implementation
of the Nesterov-Polyak method with a first-order oracle therefore has gradient
evaluation complexity $O(\epsilon^{-3/2} d)$, improving on
gradient descent only if $d \ll \epsilon^{-1/2}$, which may fail
in high-dimensions.

In two recent papers, we~\citep{carmon2016accelerated} and (independently)
\citeauthor{agarwal2016finding} obtain better rates for first-order methods.
\citet{agarwal2016finding} propose a careful implementation of the
Nesterov-Polyak method, using accelerated methods for fast approximate
matrix inversion. In our earlier work, we employ a combination of
(regularized) accelerated gradient descent and the Lanczos method. Both find
a point that satisfies the bound~\eqref{eq:approx-stationary-objective} with
probability at least $1-\delta$ using
$O\left(\epsilon^{-7/4}\log\frac{d}{\delta\epsilon}\right)$ gradient and
Hessian-vector product evaluations.

The primary conceptual difference between our approach and those of
\citeauthor{carmon2016accelerated} and \citeauthor{agarwal2016finding} is
that we perform no eigenvector search: we automatically find directions of
negative curvature whenever AGD proves $f$ ``guilty'' of
non-convexity. Qualitatively, this shows that explicit
second orders information is unnecessary
to improve upon gradient descent for stationary point
computation. Quantitatively, this leads to the following improvements:
\begin{enumerate}[(i)]
\item Our result is \emph{dimension-free and deterministic},
  with complexity independent 
  of the ratio $d / \delta$, compared to the $\log\frac{d}{\delta}$
  dependence of previous works.
  This is significant, as $\log \frac{d}{\delta}$ may be comparable
  to
  $\epsilon^{-1/4}/\log\frac{1}{\epsilon}$, making it unclear whether the 
  previous guarantees are actually better than those of gradient descent. %
\item Our method uses \emph{only gradient evaluations}, and does not require
  Hessian-vector products. In practice, Hessian-vector products may be 
  difficult to implement and more expensive to compute than gradients.
  \item 
	Under third-order smoothness assumptions we improve our method to 
	achieve $O(\epsilon^{-5/3}\log\frac{1}{\epsilon})$ rate. It is unclear how to 
	extend previous approaches to obtain similar 
	guarantees.
%
%
%
%
%
%
%
%
%
%
%
  
  %
\end{enumerate}


%
%
%
%
%
%


In distinction from the methods of \citet{carmon2016accelerated} and
\citet{agarwal2016finding}, our method provides no guarantees on positive
definiteness of $\nabla^2 f(x)$; if initialized at a 
saddle point it will terminate immediately.
However, 
as we further explain in Section~\ref{sec:combine-grad-eig},
we may combine our method with a fast eigenvector
search to recover the approximate positive definiteness guarantee $\hess f(x) 
\succeq -\sqrt{\epsilon}I$, even improving it to $\hess f(x) 
\succeq -{\epsilon}^{2/3} I$ using third-order smoothness, but
at the cost of reintroducing randomization, Hessian-vector products and a 
$\log\frac{d}{\delta}$ complexity term.

 \hide{
	This is my shot at starting the explanation with our original ``black box'' 
	intuition. I am 
	currently giving this up since it seems to take too much space.
	
	This work is based on
	(and named after) the following observation. Let $x_1, x_2, ...., x_t$ be the 
	sequence of iterates produced by a first-order method applied on the 
	function 
	$f$, and suppose there exists a $\StrConv$ strongly convex function $g$   
	that satisfies 
	$f(x_i) = g(x_i)$ and $\grad f(x_i) = \grad g(x_i)$ for $i=1,...,t$. Then if the 
	method is guaranteed to reduce the norm of $\StrConv$ strongly convex 
	functions at a certain rate, it must do so also for $\norm{\grad f(x_t)}$, 
	since 
	replacing $f$ with $g$ would not have affected the method's trajectory. 
	Conversely, suppose that 
	$\norm{\grad f(x_t)}$ fails to decrease at the prescribed rate for strongly 
	convex functions. Then, there must be two iterates $u,v\in\{x_1,...,x_t\}^2$ 
	that prove $f$ is ``guilty'' of not being $\StrConv$ strongly convex, by 
	satisfying
	\begin{equation*}
	f(v) + \grad f(v) ^T (u-v) + \frac{\StrConv}{2}\norm{u-v}^2  > f(u).
	\end{equation*}
	To see why this is so, note that if the above inequality holds for no   
	$u,v\in\{x_1,...,x_t\}^2$, then $g(x) = \max_{1\le i \le t}\left\{f(x_i) + \grad 
	f(x_i) ^T (x-x_i) + \frac{\StrConv}{2}\norm{x-x_i}^2\right\}$ is a $\StrConv$ 
	strongly convex function that agrees with $f$ and its gradient on every 
	iteration.
}

\subsection{Preliminaries and notation}
\label{sec:prelims}

Here we introduce notation and briefly overview 
definitions and results we use throughout.
We index sequences by subscripts, and use $x_i^j$ as shorthand 
for $x_i, x_{i+1}, ..., x_j$. We use $x,y,v,u,w,p,c,q$ and $z$ to denote points 
in 
$\R^\Dim$. Additionally, $\eta$ denotes step sizes, $\epsilon, \localeps$   
denote 
desired accuracy, $\theta$ denotes a scalar and  $\norm{\cdot}$
denotes the Euclidean norm on $\R^d$. We denote the $n$th derivative of a 
function $h:\R\to\R$ by $h^{(n)}$. We let $\log_+(t) = \max\{0, \log t\}$.

A function \func has $L_n$-Lipschitz $n$th derivative if it is $n$ times
differentiable and for every $x_0$ and unit vector $\ncdirection$, the
one-dimensional function $h(\theta) = f(x_0 + \theta\ncdirection)$ satisfies
\begin{equation*}
  \left|{h^{(n)}(\theta_1)-h^{(n)}(\theta_2)}\right| \le L_n |\theta_1 - \theta_2|.
\end{equation*}
We refer to this property as $n$th-order smoothness, or 
simply smoothness for $n=1$, where
it coincides with the Lipschitz continuity of $\nabla f$.
Throughout the paper, we make extensive use of the well-known consequence of
Taylor's theorem, that the Lipschitz constant of the $n$th-order derivative
controls the error in the $n$th order Taylor series expansion of $h$, \ie for 
$\theta, \theta_0 \in \R$ we have
\begin{equation}\label{eq:taylor-bound}
\left|h(\theta)-\sum_{i=0}^n \frac{1}{i!}h^{(i)}(\theta_0)(\theta-\theta_0)^i\right| 
\le 
\frac{L_n}{(n+1)!}|\theta 
- \theta_0|^{n+1}.
\end{equation}
A function $f$ is $\StrConv$-strongly convex if $f(u) \ge f(v) + \grad f(v)
^T (u-v) + \frac{\StrConv}{2}\norm{u-v}^2$ for all $v,u\in\R^\Dim$.

\section{Algorithm components}\label{sec:alg-components}

We begin our development by presenting the two building blocks of our
result: a monitored variation of AGD (Section~\ref{sub:agd-monitor}) and a
negative curvature descent step (Section~\ref{sub:exploit-nc}) that we
use when the monitored version of AGD certifies non-convexity. In
Section~\ref{sec:main-result}, we combine these components to obtain an
accelerated method for non-convex functions.

\subsection{AGD as a convexity monitor}\label{sub:agd-monitor}

The main component in our approach is Alg.~\ref{alg:AGDUPG},
\callAGDUPG{}. We take as input an $\localsmgrad$-smooth function $f$,
conjectured to be $\localstrconv$-strongly convex, and optimize it with
Nesterov's accelerated gradient descent method for strongly convex functions
(lines~\ref{line:agd-y} and~\ref{line:agd-x}). At every iteration,
the method invokes
\callCERT{} to test whether the optimization is progressing as it should for
strongly convex functions, and in particular that the gradient norm is 
decreasing
exponentially quickly (line~\ref{line:cert-w}). If the test fails, \callFWP{}
produces points $u,v$ proving that $f$ violates
$\StrConv$-strong convexity. Otherwise, we proceed until we find a point $y$
such that $\norm{\grad f(y)} \le \localeps$.

\begin{algorithm}[tb]
  \caption{
    \protect\callAGDUPG{$f,y_0$, $\localeps$, 
      $\localsmgrad$, $\localstrconv$}
    \label{alg:AGDUPG} }
  \begin{algorithmic}[1]
    \State Set $\kappa \gets \localsmgrad/\localstrconv$,  $\omega \gets 
    \frac{\sqrt{\kappa} - 1}{\sqrt{\kappa} + 1}$ and $x_0 \gets y_0$ 
    \label{line:agupg-defs}
    \For{$t =  1, 2, \ldots$}
    \State  $y_{t} \gets x_{t-1} - \frac{1}{\localsmgrad} \nabla f (x_{t-1})$ 
    \label{line:agd-y}
    \State  $x_{t} \gets y_{t} + \omega \left( y_{t} - y_{t-1} \right)$ 
    \label{line:agd-x}
    \vspace{0.05cm}
    \State $w_t \gets \callCERT{f, y_0, y_{t}, \localsmgrad, \localstrconv,\kappa}$ 
    \label{line:call-cert}
    \If { $w_t \neq \NULL$ } \Comment{convexity violation}
    \State $(u,v) \gets \callFWP{f, x_0^t, y_0^t, w_t, \localstrconv}$
    \State \Return ($x_0^{t}, y_0^{t}, u, v$) \label{line:return-witness}
    \EndIf

    \If{$\| \nabla f(y_{t}) \| \le  \localeps$}
    \Return($x_0^{t}, y_0^{t}, \NULL$)  
    \label{line:agdupg-convex-terminate}%
    \EndIf
    \EndFor
  \end{algorithmic}
  \begin{algorithmic}[1]\label{alg:CERT}
    \Function{\CERT}{$f$, $y_0$, $y_t$, $\localsmgrad$, 
      $\localstrconv$, $\kappa$}
    \If{$f(y_{t}) > f(y_0)$} \label{line:cert-yt} 
     \Return{$y_0$} \Comment{non-convex 
      behavior \arxiv{detected}}%
    \EndIf
    \State Set 
    $z_{t} \gets y_t - \frac{1}{\localsmgrad}\grad f (y_t)$\label{line:zt}
    \State Set $\psi(z_t) \gets f(y_0) - f(z_{t}) + 
    \frac{\localstrconv}{2}\norm{z_{t}-y_0}^2$ \label{line:psi-zt}
    \If{$\| \nabla f(y_{t}) \|^{2} > 2 \localsmgrad \psi(z_t)	
      e^{-t/\sqrt{\kappa}}$}
    \label{line:cert-w} 
    \Return{$z_{t}$} \Comment{AGD has stalled}
    \Else  {}
    \Return{$\NULL$}
    \EndIf
    \EndFunction
  \end{algorithmic}
  \begin{algorithmic}[1]\label{alg:FWP}
    \Function{\FWP}{$f$, $x_0^t$, $y_0^t$, $w_t$, $\localstrconv$}
    \For{$j=0,1,\ldots,t-1$}
    \For{$u=y_j, w_t$}
    \arxiv{\If{$f(u) < f(x_j) + \grad f(v)^T (u - x_j) + \frac{\localstrconv}{2} \| u - 
	x_j \|^2$}}
    \Return $(u, x_j)$
    \EndIf
    \EndFor
    \EndFor
    \State (by Corollary~\ref{coro:witness-testimony} this line is never  
    reached)
    \EndFunction
  \end{algorithmic}
\end{algorithm}

The efficacy of our method is based on the following 
guarantee on the performance of AGD.

\begin{restatable}{proposition}{propAgdConvexity}\label{prop:agd-convexity}
  Let $f$ be $\localsmgrad$-smooth, and let $y_0^t$ and $x_0^t$ be the
  sequence of iterates generated by \callAGDUPG{$f$, $y_0$, $\localsmgrad$,
    $\localeps$, $\localstrconv$} for some $\localeps>0$ and $0<\localstrconv 
  \le \localsmgrad$.  Fix
  $w \in \R ^ \Dim$. If for $s = 0,1, \ldots, t-1$ we have
  \begin{flalign}
    f(u) \ge
    f(x_s) + \nabla f(x_s)^T (u - x_s) + \frac{\StrConv}{2}
    \| u - x_s \|^2
    \label{eq:agd-convexity-req} 
  \end{flalign}
  for both $u = w$ and $u = y_s$, then
  \begin{flalign}\label{eq:agd-expected-progress}
    f(y_t) - f(w) \le \left( 1- \frac{1}{\sqrt{\kappa}} \right)^{t} \psi(w),
  \end{flalign}
  where $\kappa = \frac{\localsmgrad}{\localstrconv}$ and $\psi(w) = f(y_0)
  - f(w) + \frac{\StrConv}{2}\norm{w-y_0}^2$.
\end{restatable}
\noindent
Proposition~\ref{prop:agd-convexity} is essentially a restatement of
established results~\citep{Nesterov04, bubeck2014convex}, where we take care
to phrase the requirements on $f$ in terms of local inequalities, rather
than a global strong convexity assumption. For completeness, we provide a
proof of Proposition~\ref{prop:agd-convexity} in
Section~\ref{app:agd-convexity}.

With Proposition~\ref{prop:agd-convexity} in hand, we summarize the  
guarantees of Alg.~\ref{alg:AGDUPG} as follows.

\begin{restatable}{corollary}{coroWitnessTestimony}
  \label{coro:witness-testimony}
  Let \func be $\localsmgrad$-smooth, let $y_0 \in \R^\Dim$, $\localeps>0$ 
  and $0 < \localstrconv \le \localsmgrad$. Let $(x_0^t, y_0^t, u, v) =$ 
  \callAGDUPG{$f$, $y_0$,
    $\localeps$, $\localsmgrad$, $\localstrconv$}. Then the
  number of iterations $t$ satisfies
  \begin{equation}\label{eq:agdupg-runtime}
    t\le 1 + \max\left\{0, \sqrt{\frac{\localsmgrad}{\localstrconv}} \log \left(\frac{2 
      \localsmgrad 
      \psi(z_{t-1}) }{\localeps^2} \right)\right\},
  \end{equation}
  where $\psi(z) = f(y_0) - f(z)
  + \frac{\localstrconv}{2} \norm{z - y_0}^2$
  is as in line~\ref{line:psi-zt} of \callCERT{}. 
  If $u, v \neq \NULL$ (non-convexity was detected), then
\begin{flalign}
  f(u) < f(v) + \grad f(v)^T (u - v) + \frac{\localstrconv}{2} \| u - v \|^2
\label{eq:agd-non-convexity} 
\end{flalign}
where $v=x_j$ for some $0\le j < t$ and $u = y_j$ or $u = w_t$ (defined on 
line~\ref{line:call-cert} of \callAGDUPG{}). 
Moreover,
\begin{equation}\label{eq:agdupg-mono}
\max\{f(y_1),\ldots,f(y_{t-1}),f(u)\} \le f(y_0).
\end{equation}
\end{restatable}

\begin{proof}
  The bound~\eqref{eq:agdupg-runtime} is clear for $t = 1$. For $t>1$, the 
  algorithm has not terminated at iteration $t-1$, and so we know that neither 
  the condition in
  line~\ref{line:agdupg-convex-terminate} of \callAGDUPG{} nor the condition
  in line~\ref{line:cert-w} of \callCERT{} held at iteration $t-1$.
  Thus
  \begin{equation*}
    \localeps^2 < \norm{\grad f(y_{t-1})}^2 \le  2\localsmgrad \psi(z_{t-1})
    e^{-(t-1)/\sqrt{\kappa}},
  \end{equation*}
  which gives the bound~\eqref{eq:agdupg-runtime} when rearranged.

  Now we consider the returned vectors $x_0^t$, $y_0^t$, $u$, and $v$ from
  \callAGDUPG{}. Note that $u,v\neq\NULL$ only if $w_t \neq \NULL$. Suppose
  that $w_t = y_0$, then by line~\ref{line:cert-yt} of \callCERT{} we have,
  \begin{equation*}
    f(y_t) - f(w_t) > 0 = \left(1-\frac{1}{\sqrt{\kappa}}\right)^t \psi(w_t),
  \end{equation*}
  since $\psi(w_t) = \psi(y_0) = 0$. Since this 
  contradicts the progress bound~\eqref{eq:agd-expected-progress}, we 
  obtain the certificate~\eqref{eq:agd-non-convexity} by the contrapositive of 
  Proposition~\ref{prop:agd-convexity}:
  condition~\eqref{eq:agd-convexity-req} must not hold for some $0\le s < 
  t$, implying \callFWP{} will return for some 
  $j\le s$.

  Similarly, if $w_t=z_t = y_t - \frac{1}{\localsmgrad}\grad f(y_t)$ then by 
  line~\ref{line:cert-w} of \callCERT{} we must 
  have
  \begin{equation*}
    \frac{1}{2\localsmgrad}\norm{\grad f(y_{t})}^2 > \psi(z_t)e^{-t/\sqrt{\kappa}} 
    \ge  \left(1-\frac{1}{\sqrt{\kappa}}\right)^t \psi(z_{t}).
  \end{equation*}
  Since $f$ is $\localsmgrad$-smooth we have the
  standard progress guarantee (\emph{c.f.} 
  \citet{Nesterov04} \S 1.2.3)  
  $f(y_t) - f(z_{t}) \ge  
  \frac{1}{2\localsmgrad}\norm{\grad f(y_{t})}^2$, 
  again contradicting inequality~\eqref{eq:agd-expected-progress}. 

  To see that the bound~\eqref{eq:agdupg-mono} holds, note that $f(y_s) \le
  f(y_0)$ for $s=0,\ldots,t-1$ since condition~\ref{line:cert-yt} of
  \callCERT{} did not hold. If $u = y_j$ for some $0\le j < t$ then $f(u)
  \le f(y_0)$ holds trivially.  Alternatively, if $u_t = w_t =z_t$ then
  condition~\ref{line:cert-yt} did not hold at time $t$ as well, so we have
  $f(y_t) \le f(y_0)$ and also $f(u) = f(z_t) \le f(y_t) -
  \frac{1}{2\localsmgrad}\norm{\grad f(y_{t})}^2$ as noted above; therefore
  $f(z_{t}) \le f(y_{0})$.
\end{proof}

Before continuing, we make two remarks about implementation of
Alg.~\ref{alg:AGDUPG}.

\begin{enumerate}[(1)]
\item As stated, the algorithm requires evaluation of two function gradients
  per iteration (at $x_t$ and $y_t$).
  Corollary~\ref{coro:witness-testimony} holds essentially unchanged if
  we execute line~\ref{line:agdupg-convex-terminate} of~\callAGDUPG{} and
  lines \ref{line:zt}-\ref{line:cert-w} of~\callCERT{} only once every
  $\tau$ iterations, where $\tau$ is some fixed number (say 10). This
  reduces the number of gradient evaluations to $1+\frac{1}{\tau}$ per
  iteration.
\item Direct implementation would require $O(d\cdot t)$ memory to store the
  sequences $y_0^t$, $x_0^t$ and $\grad f(x_0^t)$ for later use by
  \callFWP{}. Alternatively, \callFWP{} can regenerate these sequences from
  their recursive definition while iterating over $j$, reducing the memory
  requirement to $O(d)$ and increasing the number of gradient and function
  evaluations by at most a factor of 2.
\end{enumerate}

In addition, while our emphasis is on applying \callAGDUPG{} to non-convex
problems, the algorithm has implications for convex optimization. For
example, we rarely know the strong convexity parameter $\StrConv$ of a given
function $f$; to remedy this, \citet{o2015adaptive} propose adaptive restart
schemes. Instead, one may repeatedly apply \callAGDUPG{} and use the
witnesses to update $\StrConv$.

\subsection{Using negative curvature}\label{sub:exploit-nc}

The second component of our approach is exploitation of negative curvature 
to decrease function values; in Section~\ref{sec:main-result} we use 
\callAGDUPG{} to generate $u, v$ such that
\begin{equation}\label{eq:exploit-pair-cond}
  f(u) < f(v) + \grad f(v)^T (u-v) - \frac{\alpha}{2}\norm{u-v}^2,
\end{equation}
a nontrivial violation of convexity (where $\alpha > 0$ is a parameter we
control using a proximal term). By taking an appropriately sized step from $u$ 
in the direction
$\pm(u-v)$, Alg.~\ref{alg:ENCP} can substantially lower
the function value near $u$
whenever the convexity violation~\eqref{eq:exploit-pair-cond} holds.
The following basic lemma shows this essential progress guarantee.

\begin{algorithm}[tb]
	\caption{\protect\callENCP{$f$, $u$, $v$, $\eta$}}\label{alg:ENCP}
	\begin{algorithmic}[1]
		\State  $\ncdirection \gets (u-v)/{\norm{u-v}}$
		\State $u_{+} \gets u + \eta\ncdirection$
		\State $u_{-} \gets u - \eta\ncdirection$
		\State \Return $\argmin_{z\in\{u_{-}, u_{+}\}} f(z)$
	\end{algorithmic}
\end{algorithm}

\begin{restatable}{lemma}{lemExploitPair}\label{lem:exploit-pair}
  Let \func have $\SmHess$-Lipschitz Hessian. Let $\alpha > 0$ and let $u$ 
  and $v$ satisfy~\eqref{eq:exploit-pair-cond}. If $\norm{u-v} \le 
  \frac{\alpha}{2\SmHess}$, then for every $\eta \le 
  \frac{\alpha}{\SmHess}$,  \callENCP{$f,u,v,\eta$} finds 
  a 
  point $z$ such that
  \begin{equation}
    f(z) \le f(u) - \frac{\alpha\eta^2}{12}.
    \label{eq:exploit-pair-progress}
  \end{equation}
\end{restatable}
\arxiv{
\begin{proof}
  We proceed in two parts; in the first part, we show that $f$ has negative
  curvature of at least $\alpha/2$ in the direction $\ncdirection = (u-v) /
  \norm{u - v}$ at the point $u$. We then show that this negative curvature
  guarantees a gradient step with magnitude $\eta$ produces the required
  progress in function value.

  Defining $\ncdirection = (u - v) / \norm{u - v}$, we obtain
  \begin{equation*}
    \int_0^{\norm{u - v}}
    \left[\int_0^\tau \ncdirection^T \hess f(v + \theta \ncdirection)
      \ncdirection d \theta\right] d\tau
    = f(u) - f(v) - \grad f(v)^T (u - v)
    \stackrel{(i)}{<}
    - \frac{\alpha}{2} \norm{u - v}^2,
  \end{equation*}
  where inequality~$(i)$ follows by assumption~\eqref{eq:exploit-pair-cond}.
  Let $\tau^* = \inf_{\theta \in [0, \norm{u-v}]}
  \ncdirection^T \hess f(v + \theta\ncdirection) \ncdirection$.
  Using that $\int_0^t \tau d\tau = t^2/2$,
  we substitute for the integrand to find that
  $\tau^* < -\alpha$, and using the Lipschitz
  continuity of $\hess f$ yields
  \begin{equation}
    \label{eqn:pair-negative-curve}
    \ncdirection^T \hess f(u) \ncdirection
    \le \tau^*
    + \SmHess \norm{u - v}
    < -\alpha + \frac{\SmHess \alpha}{2 \SmHess}
    = -\frac{\alpha}{2},
  \end{equation}
  where we have used that $\norm{u - v} \le \frac{\alpha}{2 \SmHess}$
  by assumption.

  Using again the Lipschitz continuity of $\hess f$, we 
  apply~\eqref{eq:taylor-bound} with $x_0=u, \theta_0=0$ and $\theta=\pm 
  \eta$ to obtain
  \begin{equation*}
    f(u_{\pm}) \le f(u) \pm \eta \grad f(u)^T \ncdirection
    +\frac{\eta^2}{2}\ncdirection^T 
    \hess f (u) \ncdirection + 
    \frac{\SmHess}{6}|\eta|^3.
  \end{equation*}
  The first order term must be negative for one of $u_+$ and
  $u_{-}$. Therefore, using inequality~\eqref{eqn:pair-negative-curve} 
  and 
  $\eta \le \alpha/\SmHess$, we have
  $f(z) = \min\{f(u_{+}),f(u_{-})\} \le f(u) - 
  \frac{\alpha\eta^2}{12}$ as desired.
\end{proof}
}

\section{Accelerating non-convex optimization}\label{sec:main-result}

We now combine the accelerated convergence guarantee of
Corollary~\ref{coro:witness-testimony} and the non-convex progress guarantee
of Lemma~\ref{lem:exploit-pair} to form \callMAIN{}. The idea for the
algorithm is as follows. Consider iterate $k-1$, denoted $p_{k-1}$. We
create a proximal function $\fhat $
by adding the proximal term  $\alpha \norm{x - p_{k-1}}^2$ to $f$. Applying
\callAGDUPG{} to $\fhat$ yields the sequences $x_0,\ldots,x_t$,
$y_0,\ldots,y_t$ and possibly a non-convexity witnessing pair $u,v$
(line~\ref{line:get-long-sequences}).  If $u,v$ are not available, we set
$p_k = y_t$ and continue to the next iteration. Otherwise, by
Corollary~\ref{coro:witness-testimony}, $u$ and $v$ certify that $\fhat$ is not
$\alpha$ strongly convex, and therefore that $f$ has negative curvature.
\callENCP{} then leverages this negative curvature, obtaining a point
$b^{(2)}$.  The next iterate $p_k$ is the best out of $y_0,\ldots,y_t,u$ and
$b^{(2)}$ in terms of function value.

\begin{algorithm}[tb]
	\caption{\protect{\callMAIN{$f$, $p_0$, $\SmGrad$, $\epsilon$, 
			$\alpha$, $\eta$}}\label{alg:MAIN}}
\begin{algorithmic}[1]
\For{$k =  1, 2, \ldots$}
\State Set $\fhat(x) \defeq f(x) + \alpha \norm{x - p_{k-1}}^2$
\State \label{line:get-long-sequences} $(x_0^{t}, y_0^{t}, u, v) \gets$ 
 \callAGDUPG{$\fhat$, 
$p_{k-1}$, 
$\frac{\epsilon}{10}$, 
$\SmGrad+2\alpha$, $\alpha$}
\If{$u,v=\NULL$}
$p_k \gets y_t$ \Comment{
	\arxiv{$\fhat$ was effectively strongly convex}}
\Else{} \Comment{
	\arxiv{$f$ has negative curvature on the line between $v$ and $u$}}
\State $b^{(1)} \gets \callFB{f, y_0^t, u, v}$
\State $b^{(2)} \gets \callENCP{f, u, v, \eta}$  
\State $p_k \gets \argmin_{z\in\{b^{(1)}, b^{(2)}\}}f(z)$
\EndIf
\If{$\| \nabla f(p_k) \| \le  \epsilon$}  
\Return $p_k$
\EndIf
\EndFor
\end{algorithmic}
	\begin{algorithmic}[1]\label{alg:FB}
		\Function{\FB}{$f$, $y_0^t$, $u$, $v$}
		\State \Return $\argmin_{z\in\{u,y_0,\ldots,y_t\}}f(z)$
		\EndFunction
	\end{algorithmic}
\end{algorithm}

The following central lemma provides a 
progress guarantee for each of the iterations of Alg.~\ref{alg:MAIN}.

\begin{restatable}{lemma}{lemMainProgressSecondOrder}
  \label{lem:main-progress-2nd}
  Let \func be $\SmGrad$-smooth and have $\SmHess$-Lipschitz 
  continuous Hessian, let $\epsilon, \alpha > 0$ and $p_0\in\R^\Dim$.
  Let
  $p_1,\ldots,p_K$ be the iterates \callMAIN{$f$, 
    $p_0$, 
    $\SmGrad$, $\epsilon$, $\alpha$, $\frac{\alpha}{\SmHess}$}
  generates. Then for each $k \in \{1, \ldots, K-1\}$,
  \begin{equation}\label{eq:main-prog}
    f(p_{k}) \le f(p_{k-1}) -  \min\left\{ \frac{\epsilon^2}{5\alpha}, 
    \frac{\alpha^3}{64\SmHess^2}\right\}.
  \end{equation} 
\end{restatable}
\arxiv{
\begin{proof}
  Fix an iterate index $1\le k<K$; throughout the proof we let $y_0^t, x_0^t$ 
  and 
  $u,v$ refer to outputs of \callAGDUPG{} in the $k$th iteration.  We 
  consider the cases $u,v = \NULL$ and 
  $u,v \neq \NULL$ separately. In the former case, standard proximal point 
  arguments suffice, while in the latter case we require more care.

  The simpler case is $u,v = \NULL$ (no convexity violation detected), in 
  which $p_k = y_t$ and $\|\grad \fhat(p_k)\| \le \epsilon/10$ 
  (since \callAGDUPG{} terminated on 
  line~\ref{line:agdupg-convex-terminate}). Moreover, $k<K$ implies that 
  \callMAIN{} does not terminate at iteration $k$, and therefore $\norm{\grad 
    f(p_k)} > \epsilon$. Consequently,
  \begin{equation*}
    2\alpha \norm{p_k - p_{k-1}} = \|{\grad f(p_k) - \grad \fhat(p_k)}\| \ge 
    \|{\grad f(p_k)}\| - \|{\grad \fhat(p_k)}\| \ge
    9\epsilon/10.
  \end{equation*}
  The case $u,v = \NULL$ also implies $\fhat(p_k) = \fhat(y_t) \le 
  \fhat(y_0) = f(p_{k-1})$, as the condition in line~\ref{line:cert-yt} of 
  \callCERT{} never holds, and therefore
  \begin{equation*}
    f(p_k) = \fhat(p_k) - \alpha \norm{p_k - p_{k-1}}^2 \le f(p_{k-1}) - 
    \alpha\left(\frac{9\epsilon}{20\alpha}\right)^2 \le 
    f(p_{k-1}) - 
    \frac{\epsilon^2}{5\alpha},
  \end{equation*} 
  which establishes the claim in the case $u,v = \NULL$.
  
  Now we consider the case $u,v \neq \NULL$ (non-convexity detected).
   By Corollary~\ref{coro:witness-testimony},
  \begin{equation}
    \label{eq:nc-pair}
    \fhat(u)<\fhat(v) + \grad \fhat(v)^T (u-v) +
    \frac{\alpha}{2}\norm{u - v}^2 .
  \end{equation}
  By definition of $\fhat(x) = f(x) + 
  \alpha \norm{x-y_0}^2$, we have for any $y_0, u, v \in \R^\Dim$ that
  \begin{equation*}
    f(v) + \grad f(v)^T (u-v) - \frac{\alpha}{2}\norm{u-v}^2 -  f(u)
    = \fhat(v) + 
    \grad \fhat(v)^T (u-v) + \frac{\alpha}{2}\norm{u-v}^2 - \fhat(u) > 0 ,
  \end{equation*}
  where the inequality is Eq.~\eqref{eq:nc-pair}. Therefore, we conclude 
  that inequality~\eqref{eq:exploit-pair-cond} must hold.
  To apply Lemma~\ref{lem:exploit-pair}, we must control the distance
  between $u$ and $v$. We present the following lemma, which shows how  
  $b^{(1)}$ acts as an insurance policy against $\norm{u - v}$ growing
  too large.

  \begin{lemma}
    \label{lem:progress-distance}
    Let $f$ be $\SmGrad$-smooth, and $\tau \ge 0$. 
    At any iteration of \callMAIN{}, if 
    $u,v\neq\NULL$ and the best iterate $b^{(1)}$ satisfies $f(b^{(1)}) \ge 
    f(y_0) - \alpha \tau^2$ then for $1\le i < t$,
    \begin{equation*}
      \norm{y_i - y_0} \le \tau, ~~~
      \norm{x_i - y_0} \le 3\tau,
      ~~ \mbox{and}~~ 
      \norm{u-v} \le 4\tau.
    \end{equation*}
  \end{lemma}
  \noindent
  Deferring the proof of Lemma~\ref{lem:progress-distance}
  briefly, we show how it yields Lemma~\ref{lem:main-progress-2nd}.
  We set $\tau \defeq \frac{\alpha}{8\SmHess}$ and consider two 
  cases. First, if
  \begin{equation*}
    f(b^{(1)}) \le f(y_0) - \alpha\tau^2 = f(p_{k-1}) - 
    \frac{\alpha^3}{64\SmHess^2},
  \end{equation*}
  then we are done, since $f(p_k) \le f(b^{(1)})$. In the converse
  case, if $f(b^{(1)})
  \ge f(y_0) - \alpha\tau^2$, Lemma~\ref{lem:progress-distance} implies
  $\norm{u - v} \le 4\tau \le 
  \frac{\alpha}{2\SmHess}$.
  Therefore,  we can exploit the negative
  curvature in $f$, as Lemma~\ref{lem:exploit-pair} guarantees (with 
  $\eta=\frac{\alpha}{\SmHess}$), 
  \begin{equation*}
    f(b^{(2)})  \le f(u) - 
    \frac{\alpha^3}{12\SmHess^2}
    \stackrel{(i)}{\le} f(p_{k-1}) 
    -\frac{\alpha^3}{12\SmHess^2}.
  \end{equation*}
  Here inequality~$(i)$ uses that $f(u) \le \fhat(u) \le \fhat(y_0) = 
  f(p_{k-1})$ by Corollary~\ref{coro:witness-testimony}. This 
  implies inequality~\eqref{eq:main-prog} holds and concludes the case 
  $u,v\neq\NULL$.
\end{proof}

\begin{proof-of-lemma}[\ref{lem:progress-distance}]
  We begin by noting that $\fhat(y_i) \le \fhat(y_0) = f(y_0)$ for
  $i=1,\ldots,t-1$ by
  Corollary~\ref{coro:witness-testimony}. Using $f(y_i) \ge f(b^{(1)}) \ge
  f(y_0) - \alpha \tau^2$ we therefore have
  \begin{equation*}
    \alpha \norm{y_i - y_0}^2 = \fhat(y_i) - f(y_i) \le f(y_0) - f(y_i) \le \alpha 
    \tau^2,
  \end{equation*}
  which implies $\norm{y_i - y_0} \le \tau$. By
  Corollary~\ref{coro:witness-testimony} we also have $\fhat(u) \le
  \fhat(y_0)$, so we similarly obtain $\norm{u-y_0} \le \tau$.
  Using $x_i = (1+\omega)y_i - \omega y_{i-1}$, where $\omega =
  \frac{\sqrt{\kappa} - 1}{\sqrt{\kappa} + 1} \in (0, 1)$, we have by the
  triangle inequality that
  \begin{equation*}
    \norm{x_i - y_0} \le (1+\omega)\norm{y_i - y_0} + \omega\norm{y_{i-1} - 
      y_0} \le 3 \tau
    \label{eq:x-norm-bound}
  \end{equation*}
  for every $i=1,\ldots,t-1$. Finally, since $v=x_j$ for some $0\le j \le
  t-1$, this gives the last inequality of the lemma, as
  $\norm{u-v} \le \norm{u-y_0} + \norm{x_j - y_0} \le 4\tau$.
\end{proof-of-lemma}
}

Lemma~\ref{lem:main-progress-2nd} shows
we can accelerate gradient descent in a non-convex setting. Indeed,
ignoring all problem-dependent constants, setting $\alpha = \sqrt{\epsilon}$ in
the bound~\eqref{eq:main-prog} shows that we make 
$\Omega(\epsilon^{3/2})$ progress at every iteration of~\callMAIN{}, and 
consequently the number of iterations is bounded by $O(\epsilon^{-3/2})$. 
Arguing that calls to \callAGDUPG{} each require
$O(\epsilon^{-1/4}\log\frac{1}{\epsilon})$ gradient computations yields the 
following complexity
guarantee.

\begin{restatable}{theorem}{thmMainResultSecondOrder} \label{thm:main-result-2nd}
  Let \func be $\SmGrad$-smooth and have $\SmHess$-Lipschitz 
  continuous Hessian. Let $p_0\in\R^\Dim$, $\DeltaF = f(p_0) 
  - 
  \inf_{z\in\R^\Dim}f(z)$ and
  $0 < \epsilon \le 
  \min\{\DeltaF^{2/3}\SmHess^{1/3}, \SmGrad^2/(64\SmHess)\}$. Set
  \begin{equation}\label{eq:alpha-choice}
    \alpha = 2\sqrt{\SmHess\epsilon}
  \end{equation}
  then \callMAIN{$f$, $p_0$, $\SmGrad$, $\epsilon$, $\alpha$, 
    $\frac{\alpha}{\SmHess}$} finds a point $p_K$ such that 
  $\norm{\grad{f(p_K})} \le \epsilon$ 
  with at most
  \begin{equation}\label{eq:final-bound}
    20 \cdot 
    \frac{\DeltaF\SmGrad^{1/2}\SmHess^{1/4}}{\epsilon^{7/4}}\log 
    \frac{500\SmGrad\DeltaF}{\epsilon^2}
  \end{equation}
  gradient evaluations.
\end{restatable}

\begin{proof}
  We bound two quantities: the number of calls to
  \callAGDUPG{}, which we denote by $K$, and the maximum number of 
  steps \callAGDUPG{} performs when it is called, which we denote by 
  $T$. The overall number gradient evaluations is $2KT$, as we compute 
  at most $2T$ gradients per iterations (at the points $x_0, \ldots, x_{t-1}$ and
  $y_1, \ldots, 
  y_t$).
  
  The upper bound on $K$ is immediate from 
  Lemma~\ref{lem:main-progress-2nd}, as by 
  telescoping the progress guarantee~\eqref{eq:main-prog} we obtain
  \arxiv{
    \begin{flalign*}
      \DeltaF
      \ge f(p_0) - f(p_{K-1})
      & = \sum_{k=1}^{K-1} \left(f(p_{k-1}) - 
      f(p_k)\right)
      \ge (K-1)\cdot \min\left\{ \frac{\epsilon^2}{5\alpha}, 
      \frac{\alpha^3}{64\SmHess^2}\right\} \ge 
      (K-1)\frac{\epsilon^{3/2}}{10\SmHess^{1/2}},
    \end{flalign*}
  }
  where the final inequality follows by substituting
  our choice~\eqref{eq:alpha-choice}, 
  of $\alpha$. We conclude that
  \begin{equation}\label{eq:K-bound}
    K \le 1 + 10\DeltaF\SmHess^{1/2}\epsilon^{-3/2}.
  \end{equation}
  
  To bound the number $T$ of steps
  of \callAGDUPG{}, note that for every $z\in \R ^\Dim$
  \begin{flalign*}
    \psi(z) & =  \fhat(y_0) - \fhat(z) + \frac{\alpha}{2}\norm{z-y_0}^2 
    = f(y_0) - f(z) - \frac{\alpha}{2}\norm{z-y_0}^2 \le \DeltaF.
  \end{flalign*}
  Therefore, substituting $\localeps = \epsilon/10$, $\localsmgrad = 
  \SmGrad + 
  2\alpha$ 
  and $\localsmgrad = \alpha = 2\sqrt{\SmHess \epsilon}$ into the 
  guarantee~\eqref{eq:agdupg-runtime} of 
  Corollary~\ref{coro:witness-testimony},
  \begin{equation}\label{eq:T-bound}
    T  \le 	1 + 
    \sqrt{2+\frac{\SmGrad}{2\sqrt{\SmHess\epsilon}}} 
    \log_+ \left(\frac{200(\SmGrad + 4\sqrt{\SmHess \epsilon})\DeltaF}{\epsilon^2}\right).
  \end{equation}
  
  We use $\epsilon \le \min\{\DeltaF^{2/3}\SmHess^{1/3},
  \SmGrad^2/(12\SmHess)\}$ to simplify the bounds on $K$ and $T$. Using $1
  \le \DeltaF\SmHess^{1/2}\epsilon^{-3/2}$ simplifies the
  bound~\eqref{eq:K-bound} to
  \begin{equation*}
    K \le 11\DeltaF\SmHess^{1/2}\epsilon^{-3/2}.
  \end{equation*}
  Applying $1 \le \SmGrad/(8\sqrt{\SmHess \epsilon})$ 
  in the bound~\eqref{eq:T-bound} gives
  \begin{flalign*}
    T  \le \sqrt{\frac{3}{4}} 
    \frac{\SmGrad^{1/2}}{\SmHess^{1/4}\epsilon^{1/4}} 
    \log \left(\frac{500\SmGrad\DeltaF}{\epsilon^2}\right),
  \end{flalign*}
  where $\DeltaF\SmGrad\epsilon^{-2} \ge 8$ allows us to drop the 
  subscript from the log.
  Multiplying the product of the above bounds by 2 gives the theorem.
\end{proof}

\arxiv{
  We conclude this section with two brief remarks.
  
  \begin{enumerate}[(1)]
  	\item The conditions
  	on $\epsilon$ guarantee that the bound~\eqref{eq:final-bound}
  	is non-trivial. If $\epsilon \ge \SmGrad^2 / \SmHess$,
  	then gradient descent achieves better guarantees. Indeed, with 
  	step-size $1/\SmGrad$, gradient descent satisfies $\norm{\nabla f(x)} \le 
  	\epsilon$ within at most $\frac{2\SmGrad \DeltaF}{\epsilon ^2}$ iterations  
  	\citep[cf.][Eq.~1.2.13]{Nesterov04}. Substituting $\epsilon \ge \SmGrad^2 
  	/ 
  	\SmHess$ therefore yields
  	\begin{equation*}
  	\frac{\SmGrad \DeltaF}{\epsilon^2}
  	\le \frac{\SmHess^2 \DeltaF}{\SmGrad^3}
  	= \SmHess^{1/4} \SmGrad^{1/2} \DeltaF
  	\cdot \left(\frac{\SmHess}{\SmGrad^2}\right)^\frac{7}{4}
  	\le \frac{\SmHess^{1/4} \SmGrad^{1/2} \DeltaF}{\epsilon^{7/4}}.
  	\end{equation*}
  	Alternatively, if $\epsilon > 10^{2/3} \DeltaF^{2/3} \SmHess^{1/3}$ then 
  	we 
  	have by inequality~\eqref{eq:K-bound} that $K < 2$, so Alg.~\ref{alg:MAIN} 
  	halts after at most a single iteration. Nevertheless, the 
  	bounds~\eqref{eq:K-bound} and \eqref{eq:T-bound} hold for any $\epsilon 
  	\ge 0$.

  	\item  While we state Theorem~\ref{thm:main-result-2nd} in
  	terms of gradient evaluation count, a similar bound holds for function
  	evaluations as well.  Indeed, inspection of our method reveals that each
  	iteration of Alg.~\ref{alg:MAIN} evaluates the function and not the
  	gradient at at most the three points $u, u_+$ and $u_-$; both complexity
  	measures are therefore of the same order.
  \end{enumerate}
}

\section{Incorporating third-order smoothness}\label{sec:third-order}

In this section, we show that when third-order derivatives are Lipschitz
continuous, we can improve the convergence rate of Alg.~\ref{alg:MAIN} by
modifying two of its subroutines. In Section~\ref{sub:exploit-third} we
introduce a modified version of \callENCP{} that can decrease function
values further using third-order smoothness. In
Section~\ref{sub:bound-third} we change \callFB{} to provide a guarantee
that $f(v)$ is never too large. We combine these two results in
Section~\ref{sub:improve-third} and present our improved complexity bounds.

\subsection{Making better use of negative 
curvature}\label{sub:exploit-third}

Our first observation is that third-order smoothness 
allows us to take larger 
steps and make greater progress when exploiting negative 
curvature, as the next lemma formalizes.

\begin{restatable}{lemma}{lemExploitPointCubic} \label{lem:exploit-point-cubic}
  Let \func have $\SmCubic$-Lipschitz third-order derivatives, $u \in
  \R^\Dim$, and $\ncdirection \in \R^\Dim$ be a unit vector. If
  $\ncdirection^T \grad^2 f(u) \ncdirection = -\frac{\alpha}{2} < 0$ then,
  for every $0 \le \eta \le \sqrt{3\alpha/\SmCubic}$,
  \begin{equation}
    \min\{f(u-\eta\ncdirection), f(u+\eta\ncdirection)\} \le f(u) - 
    \frac{\alpha\eta^2}{8}.
    \label{eq:exploit-point-cubic}
  \end{equation}
\end{restatable}

\begin{proof}
  For $\theta\in\R$, define $h(\theta) = f(u + \theta\ncdirection)$. By
  assumption $h'''$ is $\SmCubic$-Lipschitz continuous, and therefore
  \begin{equation*}
    h(\theta) \le h(0) + h'(0)\theta + h''(0)\frac{\theta^2}{2} + 
    h'''(0)\frac{\theta^3}{6} + \SmCubic\frac{\theta^4}{24}.
  \end{equation*} 
  Set $A_\eta = h'(0)\eta + h'''(0)\eta^3/6$ and set $\bar{\eta} = -  
  \mathrm{sign}(A_\eta) \eta$. As $h'(0)\bar{\eta}+ 
  h'''(0)\bar{\eta}^3/6 = 
  -|A_\eta| \le 0$, we have
  \begin{flalign*}
    h(\bar{\eta}) \le h(0) + h''(0)\frac{\eta^2}{2} + 
    \SmCubic\frac{\eta^4}{24} \le f(u) -  \frac{\alpha\eta^2}{8},
  \end{flalign*}
  the last inequality using $h(0) = f(u)$, $h''(0) = -\frac{\alpha}{2}$ and
  $\eta^2 \le \frac{3\alpha}{\SmCubic}$. That
  $f(u+\bar{\eta}\ncdirection) = h(\bar{\eta})$
  gives the result.
\end{proof}

Comparing Lemma \ref{lem:exploit-point-cubic} to the second part of the
proof of Lemma \ref{lem:exploit-pair}, we see that second-order smoothness
with optimal $\eta$ guarantees $\alpha^3/(12\SmHess^2)$ function decrease,
while third-order smoothness guarantees a $3\alpha^2/(8\SmCubic)$
decrease. Recalling Theorem~\ref{thm:main-result-2nd}, where $\alpha$ scales
as a power of $\epsilon$, this is evidently a significant
improvement. Additionally, this benefit is essentially free: there is no
increase in computational cost and no access to higher order
derivatives. Examining the proof, we see that the result is rooted in the
anti-symmetry of the odd-order terms in the Taylor expansion. This rules
out extending this idea to higher orders of smoothness, as they contain
symmetric fourth order terms.

Extending this insight to the setting of Lemma~\ref{lem:exploit-pair} is 
complicated by the fact that, at relevant scales of $\norm{u-v}$, it is no longer 
possible to guarantee that there is negative curvature at either $u$ or 
$v$. Nevertheless, 
we are 
able to show that a small modification of \callENCP{} achieves the required 
progress.

\begin{algorithm}[tb]
	\caption{\protect\callENCPT{$f$, $u$, $v$, $\eta$}}\label{alg:ENCPT}
	\begin{algorithmic}[1]
		\State  $\ncdirection \gets (u-v)/\norm{u-v}$
		\State  $\eta' \gets 
		\sqrt{\eta(\eta + \norm{u-v})} - 
		\norm{u-v}$ \label{line:eta-prime}
		\State $u_{+} \gets u + \eta'\ncdirection$
		\State $v_{-} \gets v - \eta\ncdirection$
		\State \Return $\argmin_{z\in\{v_{-}, u_{+}\}} f(z)$
	\end{algorithmic}
\end{algorithm}

\begin{restatable}{lemma}{lemExploitPairCubic}\label{lem:exploit-pair-cubic}
		Let \func have $\SmCubic$-Lipschitz third-order derivatives. Let $\alpha 
		> 0$ and let $u$ 
	and $v$ satisfy~\eqref{eq:exploit-pair-cond} and let $\eta \le 
	\sqrt{2\alpha/\SmCubic}$. Then for every $\norm{u-v} 
	\le \eta /2$,  \callENCPT{$f,u,v,\eta$} finds 
	a 
	point $z$ such that
	\begin{equation}	\label{eq:exploit-pair-progress-cubic}
	f(z) \le 
	\max\left\{f(v)-\frac{\alpha}{4}\eta^2,f(u)-\frac{\alpha}{12}\eta^2\right\}.
	\end{equation}
\end{restatable}

We prove Lemma~\ref{lem:exploit-pair-cubic} in
Section~\ref{app:proof-exploit-pair-cubic} in the supplementary material;
it is essentially a more technical version of the proof of
Lemma~\ref{lem:exploit-point-cubic}, where we
address the asymmetry of
condition~\eqref{eq:exploit-pair-cond} by taking steps of
different sizes from $u$ and $v$.

\subsection{Bounding the function values of the iterates using cubic 
interpolation}\label{sub:bound-third}

An important difference between Lemmas~\ref{lem:exploit-pair}
and~\ref{lem:exploit-pair-cubic} is that the former guarantees lower objective 
value than $f(u)$, while the
latter only improves $\max\{f(v),f(u)\}$. We invoke these lemmas for $v=x_j$
for some $x_j$ produced by \callAGDUPG{}, but
Corollary~\ref{coro:witness-testimony} only bounds the function value at
$y_j$ and $w$; $f(x_j)$ might be much larger than $f(y_0)$, rendering the
progress guaranteed by Lemma~\ref{lem:exploit-pair-cubic}
useless. Fortunately, we are able show that whenever this happens, there
must be a point on the line that connects $x_j, y_j$ and $y_{j-1}$ for which
the function value is much lower than $f(y_0)$. We take advantage of this
fact in Alg.~\ref{alg:FBT}, where we modify \callFB{} to consider additional
points, so that whenever the iterate it finds is not much better than $y_0$, 
then $f(x_j)$ is guaranteed to be close to $f(y_0)$. We formalize this claim
in the following lemma, which we prove in
Section~\ref{app:proof-progress-value-cubic}.

\begin{algorithm}[tb]
	\caption{\callFBT{$f$, $y_0^t$, $u$, $v$}\label{alg:FBT}}
	\begin{algorithmic}[1]
		\State Let $0\le j < t$ be such that $v=x_j$%
		\State $c_j \gets (y_j + y_{j-1})/2$ ~~\textbf{if} $j>0$ \textbf{else} $y_0$
		\State $q_j \gets -2y_i + 3y_{j-1}$ ~~\textbf{if} $j>0$ \textbf{else} $y_0$
		\State \Return $\argmin_{z\in\{y_0,...,y_t,c_j,q_j,u\}}f(z)$
	\end{algorithmic}
\end{algorithm}

\begin{restatable}{lemma}{lemProgressValueCubic} 
  \label{lem:progress-value-cubic}
  Let $f$ be $\SmGrad$-smooth and have $\SmCubic$-Lipschitz continuous
  third-order derivatives, and let $\tau \le \sqrt{{\alpha}/{(16\SmCubic)}}$
  with $\tau, \alpha, \SmGrad, \SmCubic > 0$. Consider \callMAIN{} with
  \callFB{} replaced by \callFBT{}.  At any iteration, if $u, v\neq\NULL$
  and the best iterate $b^{(1)}$ satisfies $f(b^{(1)}) \ge f(y_0) - \alpha
  \tau^2$ then, %
  \begin{equation*}
    f(v) \le  f(y_0) + 14\alpha\tau^2.
  \end{equation*}
\end{restatable}

We now explain the idea behind the proof of
Lemma~\ref{lem:progress-value-cubic}. Let $0\le j < t$ be such that $v=x_j$
(such $j$ always exists by Corollary~\ref{coro:witness-testimony}). If $j=0$
then $x_j = y_0$ and the result is trivial, so we assume $j\ge1$. Let $f_r :
\R\to\R$ be the restriction of $f$ to the line containing $y_{j-1}$ and
$y_j$ (and also $q_j, c_j$ and $x_j$). Suppose now that $f_r$ is a cubic
polynomial. Then, it is completely determined by its values at any 4 points,
and $f(x_j)=-C_1 f(q_j)+C_2 f(y_{j-1}) -C_3 f(c_j) + C_4 f(y_j)$ for $C_j
\ge 0$ independent of $f$.  By substituting the bounds $f(y_{j-1}) \vee
f(y_j) \le f(y_0)$ and $f(q_j) \wedge f(c_j) \ge f(b^{(1)}) \ge f(y_0) -
\alpha\tau^2$, we obtain an upper bound on $f(x_j)$ when $f_r$ is cubic.  To
generalize this upper bound to $f_r$ with Lipschitz third-order derivative,
we can simply add to it the approximation error of an appropriate
third-order Taylor series expansion, which is bounded by a term proportional
to $\SmCubic\tau^4 \le \alpha\tau^2/16$. 

\subsection{An improved rate of convergence}\label{sub:improve-third}

With our algorithmic and analytic upgrades established, we are ready to state 
the enhanced performance guarantees for \callMAIN{}, where from here on we 
assume that \callENCPT{} and \callFBT{} subsume \callENCP{} and \callFB{}, 
respectively.

\begin{restatable}{lemma}{lemMainProgressThirdOrder}
  \label{lem:main-progress-third}
  Let \func be $\SmGrad$-smooth and have $\SmCubic$-Lipschitz 
  continuous third-order derivatives, let $\epsilon, \alpha > 0$ and 
  $p_0\in\R^\Dim$. If 
  $p_0^K$ is the sequence of iterates produced by \callMAIN{$f$, 
    $p_0$, 
    $\SmGrad$, $\epsilon$, $\alpha$, $\sqrt{\frac{2\alpha}{\SmCubic}}$}, 
  then for 
  every $1\le k<K$,
  \begin{equation}\label{eq:main-prog-third}
    f(p_{k}) \le f(p_{k-1}) -  \min\left\{ \frac{\epsilon^2}{5\alpha}, 
    \frac{\alpha^2}{32\SmCubic}\right\}.
  \end{equation} 
\end{restatable}

The proof of Lemma~\ref{lem:main-progress-third} is essentially identical to
the proof of Lemma~\ref{lem:main-progress-2nd}, where we replace
Lemma~\ref{lem:exploit-pair} with Lemmas~\ref{lem:exploit-pair-cubic}
and~\ref{lem:progress-value-cubic} and set $\tau =
\sqrt{\alpha/(32\SmCubic)}$.  For completeness, we give a full proof in
Section~\ref{app:proof-main-progress-third}.  The gradient evaluation
complexity guarantee for third-order smoothness then follows precisely as in
our proof of Theorem~\ref{thm:main-result-2nd};
 see 
Sec.~\ref{app:proof-main-result-third} for a proof of the following

\begin{restatable}{theorem}{thmMainThird}\label{thm:main-result-third}
  Let \func be $\SmGrad$-smooth and have $\SmCubic$-Lipschitz 
  continuous third-order derivatives. Let $p_0\in\R^\Dim$, $\DeltaF = f(p_0) 
  - 
  \inf_{z\in\R^\Dim}f(z)$ and $0 < \epsilon^{2/3} \le 
  \min\{\DeltaF^{1/2}\SmCubic^{1/6}, 
  \SmGrad/(8\SmCubic^{1/3})\}$. If we set
  \begin{equation}\label{eq:alpha-choice-third}
    \alpha = 2\SmCubic^{1/3}\epsilon^{2/3},
  \end{equation}
  \callMAIN{$f$, $p_0$, $\SmGrad$, $\epsilon$, $\alpha$, 
    $\sqrt{\frac{2\alpha}{\SmCubic}}$} finds a point $p_K$ such that 
  $\norm{\grad{f(p_K})} \le \epsilon$ 
  and requires at most
  \begin{equation}\label{eq:final-bound-thirs}
    20 \cdot 
    \frac{\DeltaF\SmGrad^{1/2}\SmCubic^{1/6}}{\epsilon^{5/3}}\log \left(
    \frac{500\SmGrad\DeltaF}{\epsilon^2}\right)
  \end{equation}
  gradient evaluations.
\end{restatable}

\arxiv{
While achieving the guarantees that Theorem~\ref{thm:main-result-third} 
provides requires some modification of our algorithms, these do not
come at the expense of the convergence guarantees of
Theorem~\ref{thm:main-result-2nd} when we have only second order
smoothness. That is, the results of Theorem~\ref{thm:main-result-2nd}
remain valid even with the algorithmic modifications of this section,
and Alg.~\ref{alg:MAIN} transitions
between 
smoothness regimes by varying the scaling of $\alpha$ and $\eta$ 
with $\epsilon$.
}

%
%
%
%
%
%
%
%
%
%
%
%
%
%
%
%
%
%
%
%
%
%
%
%
%
%
%
%
%
%
%
%
%
%
%
%
%
%
%
%
%
%
%
%
%
%
%
%
%
%
%
%
%
%
%
%
%
%
%
%
%
%
%
%
%
%
%
%
%
%
%
%
%
%
%
%
%
%
%
%
%
%
%
%
%
%
%
%
%
%
%
%
%
%
%
%
%
%
%
%
%
%
%
%
%
%
%
%
%
%
%
%
%
%
%
%
%

\section{Preliminary experiments}
\label{sec:experiments}

The primary purpose of this paper is to demonstrate the feasibility of
acceleration for non-convex problems using only first-order information. 
Given the long history of development of careful schemes for non-linear 
optimization, it is unrealistic to expect a simple implementation of the 
momentum-based Algorithm~\ref{alg:MAIN} to outperform state-of-the-art 
methods such as non-linear conjugate gradients and L-BFGS. It is important, 
however, to understand the degree of non-convexity in problems we 
encounter in practice, and to investigate the efficacy of the negative curvature 
detection-and-exploitation scheme we propose.

Toward this end, we present two
experiments: (1) fitting a non-linear regression model and (2) training a small
neural network.  In these experiments we compare a basic implementation of
Alg.~\ref{alg:MAIN} with a number baseline optimization methods: gradient
descent (GD), non-linear conjugate gradients (NCG) \cite{hager2006survey},
Accelerated Gradient Descent (AGD) with adaptive restart
\cite{o2015adaptive} (RAGD), and a crippled version of Alg.~\ref{alg:MAIN}
without negative curvature exploitation (C-Alg.~\ref{alg:MAIN}). We compare
the algorithms on the number of gradient steps, but note that the number of
oracle queries per step varies between methods.  We provide implementation 
details in Section~\ref{app:implementation}.

\arxiv{
  \begin{figure}[tb]
    \centering
    \begin{minipage}[t]{0.32\textwidth}
      \centering
      \includegraphics[width=0.9\columnwidth]{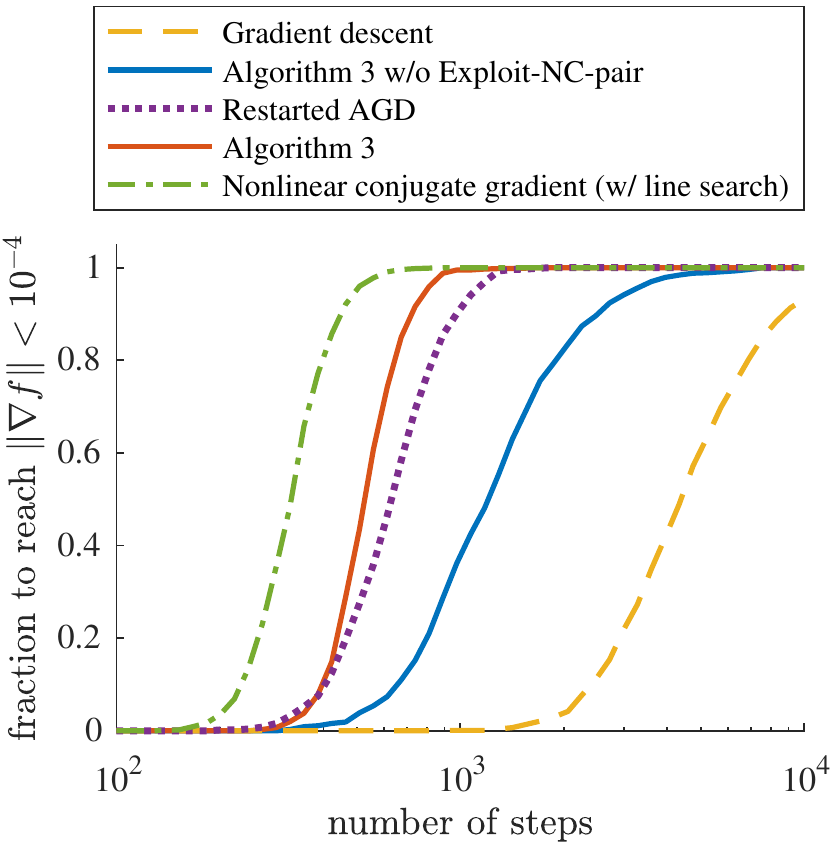}
      
      \textbf{(a)}
    \end{minipage}
    \begin{minipage}[t]{0.32\textwidth}
      \centering
      \includegraphics[width=0.9\columnwidth]{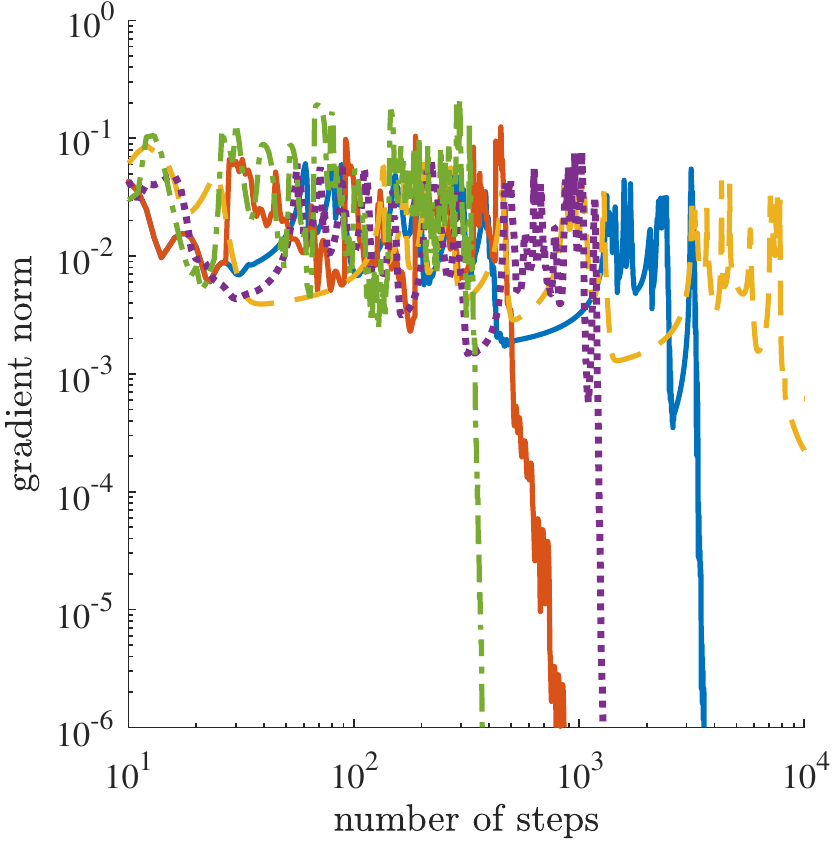}
      
      \textbf{(b)}
    \end{minipage}
    \begin{minipage}[t]{0.32\textwidth}
      \centering
      \includegraphics[width=0.9\columnwidth]{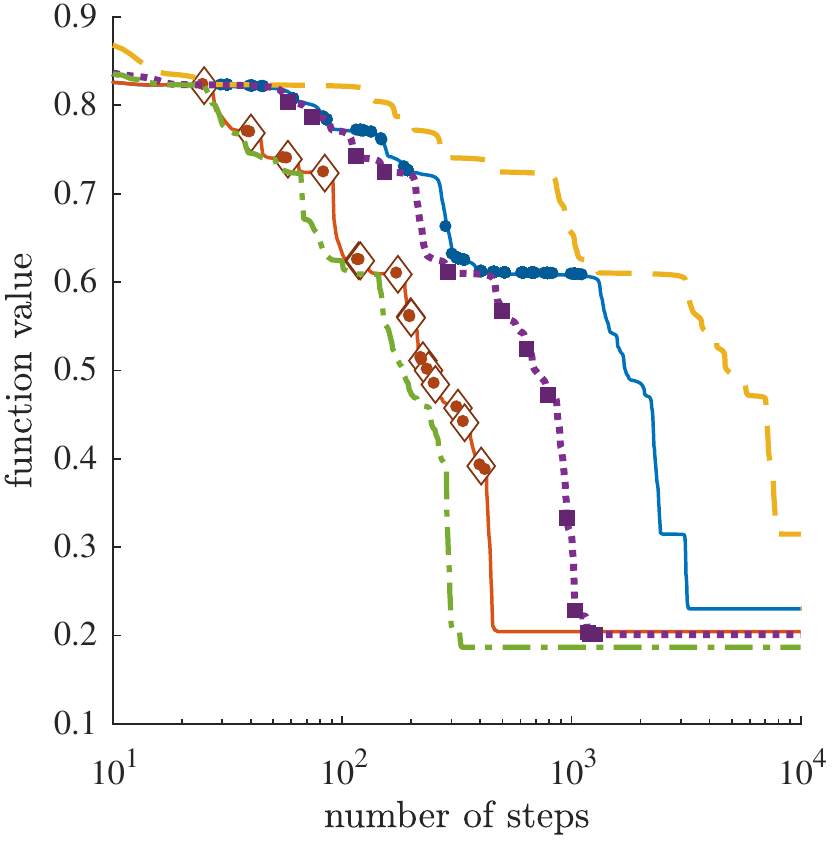}
      
      \textbf{(c)}
    \end{minipage}
    
    \caption{Performance on a non-convex regression problem. \textbf{(a)} 
      Cumulative distribution of number of steps required to achieve gradient 
      norm $<10^{-4}$, over 1,000 random problem instances \textbf{(b)} 
      Gradient norm trace for a representative instance \textbf{(c)} Function 
      value trace for the same instance. For 
      Alg.~\ref{alg:MAIN}, the dots correspond to negative curvature detection 
      and  
      the   
      diamonds correspond to negative curvature exploitation 
      (\ie when $f(b^{(2)}) < f(b^{(1)})$). For RAGD, and the squares indicate 
      restarts due to non-monotonicity. }\label{fig:tukey-problem}
  \end{figure}
}

\arxiv{

  For our first experiment, we study robust linear regression with the
  smooth biweight loss~\cite{beaton1974fitting},
  \begin{equation*}
    f(x) \defeq \frac{1}{m}
    \sum_{i=1}^{m} \phi(a_i^T x - b_i)
    ~~ \mbox{where} ~~
    \phi(\theta) = \frac{\theta^2}{1 + \theta^2}.
  \end{equation*}
  The function $\phi$ is a robust modification of the quadratic loss; it is
  approximately quadratic for small errors, but insensitive to larger
  errors.  For 1,000 independent experiments, we randomly generate problem
  data to create a highly non-convex problem as follows.  We set $\Dim = 30$
  and $m = 60$, and we draw $a_i \simiid \normal(0, I_\Dim)$. We generate
  $b$ as follows. We first draw a ``ground truth'' vector $z \sim \normal(0,
  4 I_\Dim)$.  We then set $b = A z + 3 \nu_1 + \nu_2$, where $\nu_1 \sim
  \normal(0, I_m)$ and the elements of $\nu_2$ are
  i.i.d.\ Bernoulli$(0.3)$. These parameters
  make the problem substantially non-convex.

}
In Figure~\ref{fig:tukey-problem} we plot
aggregate convergence time statistics, as well as gradient norm and
function value trajectories for a single representative problem instance.
The figure shows that gradient descent and C-Alg.~\ref{alg:MAIN}
(which does not exploit curvature) converge more
slowly than the other methods. When C-Alg.~\ref{alg:MAIN} stalls it is 
detecting negative curvature, which implies the stalling occurs around saddle 
points. When negative
curvature exploitation is enabled, Alg.~\ref{alg:MAIN} is faster than RAGD, but 
slower than NCG. In this 
highly non-convex problem, different methods often converge to local minima 
with (sometimes significantly) different function values. However, each method 
found the ``best'' local minimum in a similar fraction of the generated 
instances, so there does not appear to be a significant difference in the ability 
of the methods to find ``good'' local minima in this problem ensemble.

\begin{figure}[tb]
  \begin{center}
    \includegraphics[width=0.5\textwidth]{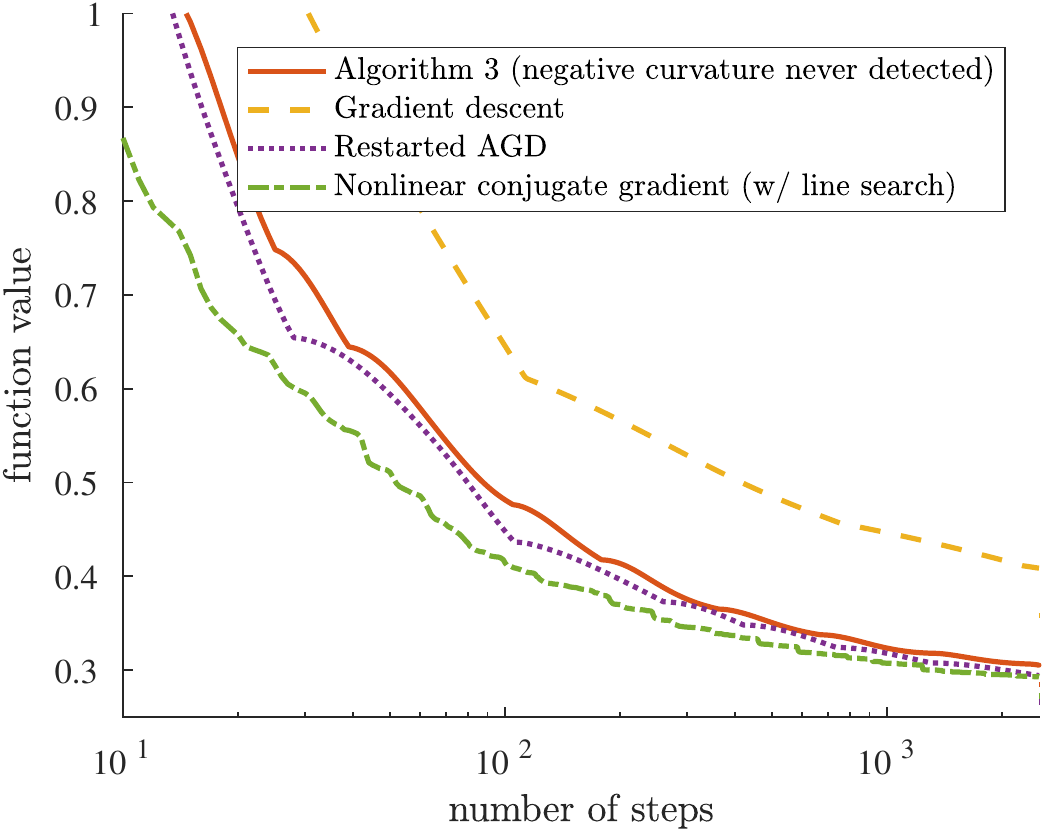}
  \end{center}
  \caption{Performance on neural network training.}\label{fig:ANN}
\end{figure}

For the second experiment we fit a neural network model\footnote{Our
  approach in its current form is inapplicable to training neural networks
  of modern scale, as it requires computation of exact gradients.}
comprising three fully-connected hidden layers containing $20$, $10$ and
$5$ units, respectively, on the MNIST handwritten digits dataset
\cite{lecun1998mnist} (see Section~\ref{app:exp-2}).
Figure~\ref{fig:ANN} shows a substantial performance gap between gradient
descent and the other methods, including Alg.~\ref{alg:MAIN}. However, this
is not due to negative curvature exploitation; in fact, Alg.~\ref{alg:MAIN}
never detects negative curvature in this problem, implying AGD never
stalls. Moreover, RAGD never restarts.  This suggests that the loss function
$f$ is ``effectively convex'' in large portions of the training trajectory,
consistent with the empirical observations
of~\citet{goodfellow2014qualitatively}; this phenomenon may
merit further investigation.

We conclude that our approach can augment AGD in the presence of negative
curvature, but that more work is necessary to make it competitive with
established methods such as non-linear conjugate gradients. For example,
adaptive schemes for setting $\alpha, \eta$ and $\SmGrad$ must be
developed. However, the success of our method may depend on whether AGD
stalls at all in real applications of non-convex optimization.

\arxiv{
  \section*{Acknowledgment}
  OH was supported by the PACCAR INC 
  fellowship. YC and JCD were partially supported by
  the SAIL-Toyota Center for AI Research and
  NSF-CAREER award 1553086. YC was partially supported 
  by the Stanford Graduate Fellowship and the Numerical Technologies 
  Fellowship.
}

\bibliographystyle{abbrvnat}
\bibliography{library-convex-until-guilty}

\newpage
\appendix
\part*{Supplementary material}
\section{Proofs from Section~\ref{sec:alg-components}}
\subsection{Proof of Proposition~\ref{prop:agd-convexity}} 
\label{app:agd-convexity}

\propAgdConvexity*

\begin{proof}
	The proof is closely based on the proof of Theorem 3.18 of 
	\cite{bubeck2014convex}, which itself is based on the estimate sequence 
	technique of \citet{Nesterov04}. We modify the proof slightly to avoid 
	arguments that depend on the global minimum of $f$. 
	This enables using inequalities \eqref{eq:agd-convexity-req} to prove the 
	result, instead of $\StrConv$-strong convexity of the function $f$.
	
	We define $\localstrconv$-strongly convex quadratic functions $\Phi_{s}$ 
	by induction as
	\begin{equation*}
	\Phi_{0}(z) = f(x_{0}) + \frac{\localstrconv}{2} \| z - x_{0} \|^2,
	\end{equation*}
	and, for $s = 0, ..., t-1$,
	\begin{flalign}\label{eq:phi-def-induct}
	\Phi_{s+1}(z) = \left( 1 - \frac{1}{\sqrt{\kappa}} \right) \Phi_{s}(z) + 
	\frac{1}{\sqrt{\kappa}} \left( f(x_{s}) + \nabla f(x_{s})^T (z - x_{s}) + 
	\frac{\localstrconv}{2} \| z - x_{s} \|^2 \right).
	\end{flalign}
	Using~\eqref{eq:agd-convexity-req} with $u=w$, straightforward induction 
	shows that
	\begin{flalign}\label{eq:Phi:lb}
	\Phi_{s}(w) \le f(w) + \left( 1 - \frac{1}{\sqrt{\kappa}} \right)^{s} 
	\psi(w)~~\mbox{for}~~s=0,1,...,t.
	\end{flalign}
	
	Let  $\Phi_{s}^{*} = \min_{x \in \R^{n}}{\Phi_{s}(x)}$. If
	\begin{flalign}\label{eq:Phi:hypothesis}
	f(y_{s}) \le \Phi_{s}^{*}~~\mbox{for}~~s=0,1,...,t
	\end{flalign}
	then \eqref{eq:agd-expected-progress} follows immediately, since
	\begin{flalign*}
	f(y_{t}) - f(w) \le \Phi_{t}^{*} - f(w) \le 
	 \Phi_{t}(w) - f(w) \le \left( 1 - \frac{1}{\sqrt{\kappa}} \right)^{t}  \psi(w)
	\end{flalign*}

We now prove \eqref{eq:Phi:hypothesis} by induction. Note that it is true at $s 
= 0$ since $x_0 = y_0$ is the global minimizer of $\Phi_0$. We have,
\begin{flalign*}
f(y_{s+1}) &\overset{\text{(a)}}{\le}  f(x_{s}) - \frac{1}{2 \localsmgrad} \| \nabla 
f 
(x_{s}) \|^2 \\
&= \left( 1 - \frac{1}{\sqrt{\kappa}} \right) f(y_{s}) + \left(1 - 
\frac{1}{\sqrt{\kappa}} \right) \left( f(x_{s}) - f(y_{s}) \right) + 
\frac{1}{\sqrt{\kappa}} f(x_{s}) - \frac{1}{2 \localsmgrad} \| \nabla f(x_{s}) \|^2 
\\
&\overset{\text{(b)}}{\le}   \left( 1 - \frac{1}{\sqrt{\kappa}} \right) \Phi_{s}^{*} + 
\left( 1 - 
\frac{1}{\sqrt{\kappa}} \right) \left( f(x_{s}) - f(y_{s}) \right) + 
\frac{1}{\sqrt{\kappa}} f(x_{s}) - \frac{1}{2 \localsmgrad} \| \nabla f(x_{s}) \|^2
\\
&\overset{\text{(c)}}{\le}   \left( 1 - \frac{1}{\sqrt{\kappa}} \right) \Phi_{s}^{*} + 
\left( 1 - 
\frac{1}{\sqrt{\kappa}} \right) \nabla f(x_{s})^T (x_{s} - y_{s} ) + 
\frac{1}{\sqrt{\kappa}} f(x_{s}) - \frac{1}{2 \localsmgrad} \| \nabla f(x_{s}) \|^2,
\end{flalign*}
where inequality (a) follows from the definition $y_{s+1} = x_s - 
\frac{1}{\localsmgrad}\grad f(x_s)$ and the $\localsmgrad$-smoothness of 
$f$, inequality (b) is the induction hypothesis and inequality (c) is 
assumption~\eqref{eq:agd-convexity-req} with $u=y_s$.

Past this point the proof is identical to the proof of Theorem 3.18 of 
\cite{bubeck2014convex}, but we continue for sake of completeness.

To complete the induction argument we now need to show that:
\begin{flalign}\label{eq:need-to-show-Phi-s}
\left( 1 - \frac{1}{\sqrt{\kappa}} \right) \Phi_{s}^{*} + \left( 1 - 
\frac{1}{\sqrt{\kappa}} \right) \nabla f(x_{s})^T (x_{s} - y_{s} ) + 
\frac{1}{\sqrt{\kappa}} f(x_{s}) - \frac{1}{2 \localsmgrad} \| \nabla f(x_{s}) \|^2 
\le \Phi_{s+1}^{*}.
\end{flalign}
Note that $\nabla^2 \Phi_{s}= \localstrconv I_{n}$ (immediate 
by induction) and therefore
\begin{equation*}
\Phi_{s} ( x ) = \Phi_{s}^{*} + \frac{\localstrconv}{2} \| x - v_{s} \|^2 ,
\end{equation*}
for some $v_{s} \in \R^{n}$. By differentiating 
\eqref{eq:phi-def-induct} and using the above form of $\Phi_{s}$ we obtain
\begin{equation*}
\nabla \Phi_{s+1} ( x ) = \localstrconv \left( 1 - \frac{1}{\sqrt{\kappa}} \right) (x 
- v_{s} ) + \frac{1}{\sqrt{\kappa}} \nabla f(x_{s}) + 
\frac{\localstrconv}{\sqrt{\kappa}} ( x - x_{s} ).
\end{equation*}
Since by definition $\Phi_{s+1}(v_{s+1})=0$, we have
\begin{flalign}\label{vs-def-induct}
v_{s+1} = \left( 1 - \frac{1}{\sqrt{\kappa}} \right) v_{s} + 
\frac{1}{\sqrt{\kappa}} x_{s} - \frac{1}{\localstrconv \sqrt{\kappa}} \nabla 
f(x_{s})
\end{flalign}
Using \eqref{eq:phi-def-induct}, evaluating evaluating $\Phi_{s+1}$ at $x_{s}$ 
gives,
\begin{flalign}\label{eq:Phi-s-at-xs}
\Phi_{s+1}(x_s) = \Phi_{s+1}^{*}  + \frac{\localstrconv}{2} \| x_{s} - v_{s+1} 
\|^2 = \left( 1 - 
\frac{1}{\sqrt{\kappa}} \right)\left[ \Phi_{s}^{*} + \frac{\localstrconv}{2} \| x_{s} 
- v_{s} \|^2\right] + \frac{1}{\sqrt{\kappa}} 
f(x_{s}).
\end{flalign}
Substituting  \eqref{vs-def-induct} gives
\begin{equation*}
\| x_{s} - v_{s+1} \|^2 = \left( 1 - \frac{1}{\sqrt{\kappa}} \right)^2 \| x_{s} - 
v_{s} 
\|^2 + \frac{1}{\localstrconv^2 \kappa} \| \nabla f(x_{s}) \|^2 - 
\frac{2}{\localstrconv \sqrt{\kappa}} \left( 1 - \frac{1}{\sqrt{\kappa}} \right) 
\nabla f(x_{s})^T (v_{s} - x_{s})
\end{equation*}
which combined with \eqref{eq:Phi-s-at-xs} yields
\begin{flalign*}
\Phi_{s+1}^{*} = \left( 1 - \frac{1}{\sqrt{\kappa}} \right) \Phi_{s}^{*}  
 + \frac{1}{\sqrt{\kappa}}\left( 1 - 
\frac{1}{\sqrt{\kappa}} \right) \nabla f(x_{s})^T (v_{s}  - x_{s} )  + 
\frac{1}{\sqrt{\kappa}} f(x_{s}) - \frac{1}{2 \localsmgrad} \| \nabla f(x_{s}) \|^2 
\\ +
\frac{\localstrconv}{2 \sqrt{\kappa}} \left( 1 - \frac{1}{\sqrt{\kappa}} \right) \| 
x_{s} - v_{s} \|^2 .
\end{flalign*}
Examining this equation, it is seen that $v_{s} - x_{s} = \sqrt{\kappa} (x_{s} - 
y_{s})$ implies~ \eqref{eq:need-to-show-Phi-s} and therefore concludes the 
proof of Proposition~\ref{prop:agd-convexity}. We establish the relation $v_{s} 
- x_{s} = \sqrt{\kappa} (x_{s} - y_{s})$ by induction,
\begin{flalign*}
v_{s+1} - x_{s+1} &= \left( 1 - \frac{1}{\sqrt{\kappa}} \right) v_{s} + 
\frac{1}{\sqrt{\kappa}} x_{s} - \frac{1}{\localstrconv \sqrt{\kappa}} \nabla 
f(x_{s}) - x_{s+1} \\
&= \sqrt{\kappa} x_{s} - ( \sqrt{\kappa} - 1 ) y_{s} - 
\frac{\sqrt{\kappa}}{\localsmgrad} \nabla f(x_{s}) - x_{s+1} \\
&= \sqrt{\kappa} y_{s+1} - ( \sqrt{\kappa} - 1) y_{s} - x_{s+1} 
= \sqrt{\kappa} (x_{s+1} - y_{s+1} ).
\end{flalign*}
where the first equality comes from \eqref{vs-def-induct}, the second from the induction hypothesis, the third from the definition of $y_{s+1}$ and the last one from the definition of $x_{s+1}$.
\end{proof}


\section{Proofs from Section~\ref{sec:third-order}}

\subsection{Proof of Lemma~\ref{lem:exploit-pair-cubic}}
\label{app:proof-exploit-pair-cubic}

We begin by proving the following normalized version of 
Lemma~\ref{lem:exploit-pair-cubic}.

\begin{restatable}{lemma}{lemCubicApproxNc} \label{lem:cubic-approx-nc}
	Let $h:\R \to \R$ be thrice differentiable, $h'''$ be $L$-Lipschitz 
	continuous for some $L>0$ and let
	\begin{equation}\label{eq:cubic-approx-nc-conds}
	h(1)-h(-1)-2h'(-1)=\int_{-1}^{1} 
	d\nu{\int_{-1}^{\nu}{h''(\xi) d\xi}} \le -A,
	\end{equation}
	for some $A \ge 0$. Then for any $\rho \ge 4$
	\begin{equation}\label{eq:cubic-approx-nc-guarantee}
	\min\{h(-1-\rho), h(1+\rho')\} \le \max\left\{h(-1) - \frac{A}{4}\rho^2,h(1) - 
	\frac{A}{6} \rho^2\right\}  + \frac{L}{8}\rho^4,
	\end{equation}
	where $\rho' = \sqrt{\rho(\rho+2)}-2$.
\end{restatable}

\begin{proof}
	Define
	\begin{equation*}
	\tilde{h}(\xi) = h(0) + h'(0)\xi + \frac{1}{2}h''(0)\xi^2 + \frac{1}{6}h'''(0)\xi^3.
	\end{equation*}
	By the Lipschitz continuity of $h'''$, we have that 
	$|h(\xi)-\tilde{h}(\xi)| \le L\xi^4/24$  for any $\xi \in \R$ (see Section~\ref{sec:prelims}). Similarly, viewing $h'''$ as the first derivative of $h''$, we have and $|h''(\xi) - \tilde{h}''(\xi)| \le 
	L\xi^2/2$. The 
	assumption~\eqref{eq:cubic-approx-nc-conds} therefore implies,
	\begin{equation}\label{eq:htilde-nc}
	4\left[\frac{1}{2}h''(0)-\frac{1}{6}h'''(0)\right] = \int_{-1}^{1} 
	d\nu{\int_{-1}^{\nu}{\tilde{h}''(\xi) d\xi}} \le -A + \frac{L}{2}\int_{-1}^{1} 
	d\nu{\int_{-1}^{\nu}{\xi^2 d\xi}} = -A + \frac{1}{3}L.
	\end{equation}
	It is also easy to verify that
	\begin{equation*}
	h(0) = \frac{\tilde{h}(1)+\tilde{h}(-1)}{2} - \frac{1}{2}h''(0)
	~~\mbox{and}~~
	h'(0) = \frac{\tilde{h}(1)-\tilde{h}(-1)}{2} - \frac{1}{6}h'''(0).	
	\end{equation*}
	Substituting into the definition of $\tilde{h}$ and rearranging, this yields
	\begin{equation}\label{eq:htilde-plus-rho}
	\tilde{h}(1+\rho') = \tilde{h}(1) + \frac{\tilde{h}(1)-\tilde{h}(-1)}{2}\rho' 
	+ \left[ \frac{1}{2}h''(0) - \frac{1}{6}h'''(0)\right]\rho'(2+\rho') 
	+ \frac{1}{6}h'''(0)\rho' (2+\rho')^2
	\end{equation}
	and
	\begin{equation}\label{eq:htilde-minus-rho}
	\tilde{h}(-1-\rho) = \tilde{h}(-1) - \frac{\tilde{h}(1)-\tilde{h}(-1)}{2}\rho 
	+ \left[ \frac{1}{2}h''(0) - \frac{1}{6}h'''(0)\right]\rho(2+\rho) 
	- \frac{1}{6}h'''(0)\rho^2 (2+\rho).
	\end{equation}
	
	Suppose that $\frac{\tilde{h}(1)-\tilde{h}(-1)}{2} + 
	\frac{1}{6}h'''(0)\rho(2+\rho) \ge 0$. By 
	\eqref{eq:htilde-minus-rho},~\eqref{eq:htilde-nc} and $\rho^2 \le 
	\rho(2+\rho) \le 3\rho^2/2$ (since $\rho \ge 4$) we
	then have
	\begin{equation*}
	\tilde{h}(-1-\rho) \le \tilde{h}(-1) - \frac{A}{4}\rho^2 + \frac{L}{8}\rho^2,
	\end{equation*}
	and, using $|h(\xi)-\tilde{h}(\xi)| \le L\xi^4/24$ for $\xi = -1$ and $\xi = 
	-1-\rho$ along with $1 \le \rho/4$, we get
	\begin{equation}\label{eq:cubic-nc-case-1}
	h(-1-\rho) \le h(-1) - \frac{A}{4}\rho^2 + 
	\frac{L}{8}\rho^2 + 
	\frac{L}{24}(1+(1+\rho)^4) \le h(-1)
	- \frac{A}{4}\rho^2  +\frac{L}{8}\rho^4.
	\end{equation}
	
	Suppose now that $\frac{\tilde{h}(1)-\tilde{h}(-1)}{2} + 
	\frac{1}{6}h'''(0)\rho(2+\rho) < 0$ holds instead.  By 
	\eqref{eq:htilde-plus-rho} and \eqref{eq:htilde-nc} we 
	then have
	\begin{flalign*}
	\tilde{h}(1+\rho') & \le \tilde{h}(1) - \left[ \frac{A}{4}-\frac{L}{12} 
	\right]\rho'(2+\rho') + \frac{1}{6}h'''(0)\rho'\left[ (2+\rho')^2 - \rho(2+\rho)] 
	\right]
	= \tilde{h}(1) - \left[ \frac{A}{4}-\frac{L}{12}\right]\rho'(2+\rho')
	\end{flalign*}
	where the equality follows from the definition $(2+\rho')^2 = 
	\rho(2+\rho)$. We lower bound $\rho'(2+\rho')$ as
	\begin{equation*}
	\rho'(2+\rho') = \rho(2+\rho) - 2\sqrt{\rho(2+\rho)} \ge 
	\rho\left(\frac{\rho}{2}+\rho\right) - 
	\frac{\rho}{2}\sqrt{\rho\left(\frac{\rho}{2}+\rho\right)} \ge \frac{2\rho^2}{3},
	\end{equation*} %
	where the first inequality follows from the fact that $\rho(\zeta+\rho) - 
	\zeta \sqrt{\rho(\zeta+\rho)}$ is monotonically 
	decreasing in $\zeta \ge 0$ and the assumption $2 \le \rho/2$. Noting that 
	$\rho' \le \rho$, we have the upper bound $\rho'(2+\rho') \le \rho(2+\rho) 
	\le 3\rho^2/2$. Combining these bounds gives $\tilde{h}(1+\rho') \le 
	\tilde{h}(1) - 
	\frac{A}{6}\rho^2 + \frac{L}{8}\rho^2$. Applying $|h(\xi)-\tilde{h}(\xi)| \le 
	L\xi^4/24$ at $\xi = 1$ and $\xi = 
	1+\rho'$, and using $\rho' \le \rho$  and $1\le\rho/4$ once more, we 
	obtain,
	\begin{equation}\label{eq:cubic-nc-case-2}
	h(1+\rho') \le h(1) - \frac{A}{6}\rho^2 + \frac{L}{8}\rho^2 + 
	\frac{L}{24}(1+ (1 + \rho)^4) \le 
	h(1) - \frac{A}{6}\rho^2  +\frac{L}{8}\rho^4.
	\end{equation}
	The fact that either~\eqref{eq:cubic-nc-case-1} 
	or~\eqref{eq:cubic-nc-case-2} must hold implies 
	\eqref{eq:cubic-approx-nc-guarantee}.

\end{proof}

With the auxiliary Lemma~\ref{lem:cubic-approx-nc}, we prove 
Lemma~\ref{lem:exploit-pair-cubic}.

\lemExploitPairCubic* 

\begin{proof}
	Define
	\begin{equation*}
	h(\theta) \defeq f\left( \frac{1+\theta}{2}u + \frac{1-\theta}{2}v\right).
	\end{equation*}
	We have
	\begin{equation*}
	h(1)-h(-1)-2h'(-1) = f(u) - f(v) - \grad f(v)^T(u-v) < 
	-\frac{\alpha}{2}\norm{u-v}^2\defeq -A.
	\end{equation*}
	Additionally, since $f$ has $\SmCubic$-Lipschitz third order derivatives, 
	$h'''$ is $\frac{1}{16}\SmCubic\norm{u-v}^4 \defeq L$ Lipschitz 
	continuous, so we 
	may apply Lemma~\ref{lem:cubic-approx-nc} at $\rho = 2\eta/\norm{u-v} 
	\ge 4$. Letting $\ncdirection = (u-v)/\norm{u-v}$, we note that $h(1-\rho) = 
	f(v- \eta \ncdirection)$. Similarly, for $2 + \rho' = \sqrt{\rho(2+\rho)}$ we 
	have $h(1+\rho') = f(u + \eta' \ncdirection)$ with $\eta'$ given in 
	line~\ref{line:eta-prime} of \callENCPT{}. The result is now immediate 
	from~\eqref{eq:cubic-approx-nc-guarantee}, as
	\begin{flalign*}
	f(z) &= \min\{f(v-\eta\ncdirection),f(u+\eta'\ncdirection)\} = 
	\min\{h(-1-\rho),h(1+\rho')\} \le \max\left\{h(-1)-\frac{A}{4}\rho^2,h(1) 
	-\frac{A}{6}\rho^2\right\}  + 	\frac{L}{8}\rho^4
	\\&
	= \max\left\{f(v)-\frac{\alpha}{2}\eta^2,f(u)-\frac{\alpha}{3}\eta^2\right\} + 
	\frac{\SmCubic}{8}\eta^4 \le 
	\max\left\{f(v)-\frac{\alpha}{4}\eta^2,f(u)-\frac{\alpha}{12}\eta^2\right\},
	\end{flalign*}
	where in the last transition we have used $\eta \le 
	\sqrt{\frac{2\alpha}{\SmCubic}}$.
\end{proof}

\subsection{Proof of 
Lemma~\ref{lem:progress-value-cubic}}\label{app:proof-progress-value-cubic}

We first state and prove a normalized version of the central argument in the 
proof of Lemma~\ref{lem:progress-value-cubic}

\begin{lemma} \label{lem:cubic-approx}
	Let $h:\R \to \R$ be thrice differentiable and let $h'''$ be $L$-Lipschitz 
	continuous for some $L>0$. If
	\begin{equation}\label{eq:cubic-approx-reqs}
	h(0) \le A, h(-1/2) \ge -B, h(-1) \le C~\mbox{and}~h(-3) \ge -D
	\end{equation}
	for some $A,B,C,D \ge 0$, then
	\begin{equation*}
	h(\theta) \le h(0) + 7A + 12.8B + 6C + 0.2D + L
	\end{equation*}
	for any $\theta \in [0,1]$.
\end{lemma}

\begin{proof}
	Define
	\begin{equation*}
	\tilde{h}(\xi) = h(0) + h'(0)\xi + \frac{1}{2}h''(0)\xi^2 + \frac{1}{6}h'''(0)\xi^3.
	\end{equation*}
	By the Lipschitz continuity of $h'''$, we have that 
	$|h(\xi)-\tilde{h}(\xi)| \le L\xi^4/24$, for any $\xi \in \R$. Using 
	the expressions for $\tilde{h}(x)$ at $\xi=-3,-1,-1/2$ to eliminate $h'(0),  
	h''(0)$ and $h'''(0)$, we obtain:
	\begin{flalign*}
	\tilde{h}(\theta) = h(0) & -\tilde{h}(-3)\left[ \frac{1}{30}\theta 
	+\frac{1}{10}\theta^2 + 
	\frac{1}{15}\theta^3\right]
	+ \tilde{h}(-1)\left[ \frac{3}{2}\theta +\frac{7}{2}\theta^2 +\theta^3\right]
	\\&
	- \tilde{h}(-1/2)\left[ \frac{24}{5}\theta +\frac{32}{5}\theta^2 
	+\frac{8}{5}\theta^3\right]
	+ {h}(0)\left[ \frac{10}{3}\theta +3\theta^2 +\frac{2}{3}\theta^3\right].
	\end{flalign*}
	Applying~\eqref{eq:cubic-approx-reqs}, $\theta \in [0,1]$ and 
	$|h(\xi)-\tilde{h}(\xi)| \le L\xi^4/24$ gives the required bound:
	\begin{flalign*}
	h(\theta) &\le h(0) + 0.2D + 6C + 12.8B + 7A + \frac{L}{24}\left[ \theta^4 + 
	0.2\cdot(-3)^4 + 6\cdot (-1)^4 + 12.8\cdot(-1/2)^4 \right]
	\\ &
	\le h(0) + 7A + 12.8B + 6C + 0.2D + L
	\end{flalign*}
\end{proof}

We now prove Lemma~\ref{lem:progress-value-cubic} itself.

\lemProgressValueCubic*

\begin{proof}
	 Let $0\le j < t$ be such that $v=x_j$ 
	(such $j$ always exists by Corollary~\ref{coro:witness-testimony}). If $j=0$ 
	then $x_j = y_0$ and the result is trivial, so we assume $j\ge1$. Let
	\begin{equation*}
	h(\theta) = f(y_j + \theta (y_j - y_{j-1}))-f(y_0)~~\mbox{for}~~\theta\in\R
	\end{equation*}
	Note that 
\begin{flalign*}
	h(-3) &= f(q_j) - f(y_0) \ge f(b^{(1)}) -f(y_0) \ge -\alpha\tau^2,\\ 
	h(-1) &= f(y_{j-1}) - f(y_0) \le 0, \\ 
	h(-1/2) &= f(c_j) - f(y_0) \ge f(b^{(1)}) -f(y_0) \ge -\alpha\tau^2, \\
	h(0) &= f(y_j)- f(y_0) \le 0 ~~\mbox{and} \\
	h(\omega) &= f(x_j) - f(y_0),
\end{flalign*}
	 where $0 < \omega < 1$ is defined in 
	line~\ref{line:agupg-defs} of \callAGDUPG{}, and we have used the 
	guarantee  $\max\{ f(y_{j-1}), f(y_j) \} \le f(y_0)$ from Corollary~\ref{coro:witness-testimony}. 
	Moreover, by the Lipschitz 
	continuity of the third derivatives of $f$, $h'''$ is 
	$\SmCubic\norm{y_j-y_{j-1}}^4$-Lipschitz continuous. Therefore, we can 
	apply Lemma~\ref{lem:cubic-approx} with $A=C=0$ and 
	$B=D=\alpha\tau^2$ at $\theta=\omega$ and obtain
	\begin{equation*}
	f(v) - f(y_0) =  f(x_j) - f(y_0) \le f(y_j) - f(y_0) + 13\alpha\tau^2 + 
	\SmCubic\norm{y_j-y_{j-1}}^4 \le 13\alpha\tau^2 + 
	\SmCubic\norm{y_j-y_{j-1}}^4.
	\end{equation*}
	To complete the proof, we note that Lemma~\ref{lem:progress-distance} 
	guarantees $\norm{y_j-y_{j-1}} \le \norm{y_j-y_0}+\norm{y_{j-1}-y_0} \le 
	2\tau$ and therefore
	\begin{equation*}
	\SmCubic\norm{y_j-y_{j-1}}^4 \le 16\SmCubic\tau^4 \le \alpha\tau^2,
	\end{equation*}
	where we have used $\tau^2 \le \alpha/(16\SmCubic)$.
\end{proof}

\subsection{Proof of 
Lemma~\ref{lem:main-progress-third}}\label{app:proof-main-progress-third}

\lemMainProgressThirdOrder*

\begin{proof}
	Fix an iterate index $1\le k<K$; throughout the proof we let $y_0^t, x_0^t$ 
	and $w$ refer to outputs of \callAGDUPG{} in the $k$th iteration.  We 
	consider only the case $v,u \neq \NULL$, as the argument for 
	$v,u=\NULL$ 
	is unchanged from Lemma~\ref{lem:main-progress-2nd}.
	
	As argued in the proof of Lemma~\ref{lem:main-progress-2nd}, when $v, u
	\neq \NULL$, 
	condition~\eqref{eq:exploit-pair-cond} holds. We set 
	$\tau \defeq \sqrt{\frac{\alpha}{32\SmCubic}}$ and consider two 
	cases. First, if $f(b^{(1)}) \le f(y_0) - \alpha\tau^2 = f(p_{k-1}) - 
	\frac{\alpha^2}{32\SmCubic}$ then we are done, since $f(p_k) \le 
	f(b^{(1)})$. Second, if $f(b^{(1)}) \ge f(y_0) - \alpha\tau^2$, by 
	Lemma~\ref{lem:progress-distance} we have that
	\begin{equation*}
	\norm{v-u} \le4\tau \le 
	\sqrt{\frac{\alpha}{2\SmCubic}} = \frac{\eta}{2},
	\end{equation*}
	Therefore, we can use Lemma~\ref{lem:exploit-pair-cubic} (with 
	$\eta$ as defined above) to show that
	\begin{equation}\label{eq:fb2-cubic}
	f(b^{(2)}) \le  \max\left\{ f(v) - 
	\frac{\alpha^2}{2\SmCubic}, f(u) - \frac{\alpha^2}{6\SmCubic}\right\}.
	\end{equation}
	By Corollary~\ref{coro:witness-testimony}, $f(u) \le \fhat(u) \le \fhat(y_0) = 
	f(p_{k-1})$. Moreover, since $f(b^{(1)}) \ge f(y_0) - 
	\alpha\tau^2$ and $\tau = \sqrt{\frac{\alpha}{32\SmCubic}}$, we may 
	apply 
	Lemma~\ref{lem:progress-value-cubic} to obtain
	\begin{equation*}
	f(v) \le f(y_0) + 14\alpha\tau^2 \le f(p_{k-1}) + 
	\frac{7\alpha^2}{16\SmCubic}.
	\end{equation*}
	Combining this with~\eqref{eq:fb2-cubic}, we find that
	\begin{equation*}
	f(p_k) \le f(b^{(2)}) \le f(p_{k-1}) - \min\left\{ \frac{\alpha^2}{2\SmCubic} - 
	\frac{7\alpha^2}{16\SmCubic}, \frac{\alpha^2}{6\SmCubic}\right\} = 
	f(p_{k-1}) - \frac{\alpha^2}{16\SmCubic},
	\end{equation*}
	which concludes the case $v,u\neq \NULL$ under third-order smoothness.
\end{proof}
\subsection{Proof of Theorem~\ref{thm:main-result-third}}
\label{app:proof-main-result-third}

\thmMainThird*

\begin{proof}
	The proof proceeds exactly like the proof of 
	Theorem~\ref{thm:main-result-2nd}. As argued there, the number of 
	gradient evaluations is at most $2KT$, where $K$ is number of iterations of 
	\callMAIN{} and $T$ is the maximum amount of steps performed in any 
	call to \callAGDUPG{}. 
	
	We derive the upper bound on $K$ directly from 
	Lemma~\ref{lem:main-progress-third}, as by 
	telescoping~\eqref{eq:main-prog} we obtain
	\begin{flalign*}
	\DeltaF &\ge f(p_0) - f(p_{K-1}) = \sum_{k=1}^{K-1} \left(f(p_{k-1}) - 
	f(p_k)\right) \ge (K-1)\cdot \min\left\{ \frac{\epsilon^2}{5\alpha}, 
	\frac{\alpha^2}{32\SmCubic}\right\} \ge 
	(K-1)\frac{\epsilon^{4/3}}{10\SmCubic^{1/3}},
	\end{flalign*}
	where the last transition follows from 
	substituting~\eqref{eq:alpha-choice-third}, 
	our choice of $\alpha$. We therefore conclude that
	\begin{equation}\label{eq:K-bound-third}
	K \le 1 + 10\DeltaF\SmCubic^{1/3}\epsilon^{-4/3}.
	\end{equation}
	
	To bound $T$, we recall that $\psi(z) \le \DeltaF$ for every $z\in\R^\Dim$, 
	as argued in the proof Theorem~\ref{thm:main-result-2nd}. 
	Therefore, substituting $\localeps = \epsilon/10$, $\localsmgrad = 
	\SmGrad + 
	2\alpha$ 
	and $\localstrconv = \alpha = 2{\SmCubic^{1/3} \epsilon^{2/3}}$ into the 
	guarantee~\eqref{eq:agdupg-runtime} of 
	Corollary~\ref{coro:witness-testimony} we obtain,
	\begin{equation}\label{eq:T-bound-third}
	T  \le 	1 + 
	\sqrt{2+\frac{\SmGrad}{2\SmCubic^{1/3} \epsilon^{2/3}}}
	\log_+ \left(\frac{200(\SmGrad + 4\SmCubic^{1/3} 
	\epsilon^{2/3})\DeltaF}{\epsilon^2}\right),
	\end{equation}
	where $\log_+(\cdot)$ is shorthand for $ \max\{0,\log(\cdot)\}$.
	
	Finally, we use $\epsilon^{2/3} \le 
	\min\{\DeltaF^{1/2}\SmCubic^{1/6}, 
	\SmGrad/(8\SmCubic^{1/3})\}$ to 
	simplify 
	the bounds on $K$ and $T$. Using $1 \le 
	\DeltaF\SmCubic^{1/3}\epsilon^{-4/3}$ reduces~\eqref{eq:K-bound} to
	\begin{equation*}
	K \le 11\DeltaF\SmCubic^{1/3}\epsilon^{-4/3}.
	\end{equation*}
	Applying $1 \le \SmGrad/(8{\SmCubic^{1/3} \epsilon^{2/3}})$ 
	on~\eqref{eq:T-bound} gives
	\begin{flalign*}
	T  \le \sqrt{\frac{3}{4}} 
	\frac{\SmGrad^{1/2}}{\SmCubic^{1/6}\epsilon^{1/3}} 
	\log\frac{500\SmGrad\DeltaF}{\epsilon^2},
	\end{flalign*}
	where $\DeltaF\SmGrad\epsilon^{-2} \ge 8$ allows us to drop the 
	subscript from the log. Multiplying the product of the above bounds by 2 
	gives the result.
\end{proof}

\section{Adding a second-order guarantee}\label{sec:combine-grad-eig}

\newcommand{\callEIG}[1]{{\Call{Approx-Eig}{#1}}}
\newcommand{\Ot}{\wt{O}}

In this section, we sketch how to obtain simultaneous guarantees on the 
gradient and minimum eigenvalue of the Hessian. We use the $\Ot(\cdot)$ 
notation to hide logarithmic dependence on $\epsilon$, Lipschitz 
constants $\DeltaF, \SmGrad, \SmHess, \SmCubic$ and a high probability 
confidence parameter $\delta\in(0,1)$, as well as lower order polynomial 
terms in $\epsilon^{-1}$.

Using approximate 
eigenvector computation, we can efficiently generate a direction of negative 
curvature, 
unless the Hessian is almost positive semi-definite. More explicitly, 
there exist methods of the form \callEIG{$f$, $x$, $\SmGrad$, $\alpha$, 
$\delta$} that 
require 
$\Ot(\sqrt{\SmGrad/\alpha}\log d)$ Hessian-vector products to produce a unit 
vector $v$ such that whenever 
$\hess f(x) \succeq -\alpha I$, with probability at least $1-\delta$ we have  
$v^T \hess f(x) v \le -\alpha/2$, \emph{e.g.} the Lanczos method (see 
additional 
discussion in \citep[\S 2.2]{carmon2016accelerated}). 
Whenever a unit vector $v$ satisfying $v^T \hess f(x) v 
\le -\alpha/2$ is available, we can use it to make function progress. If $\hess 
f$ is $\SmHess$-Lipschitz continuous then by Lemma~\ref{lem:exploit-pair}  
$f(x\pm\frac{\alpha}{\SmHess}v) < f(x) - \frac{\alpha^3}{12\SmHess^2}$ 
where by 
$f(x\pm z)$ we mean $\min\{f(x+z), f(x-z)\}$. If instead $f$ has 
$\SmCubic$-Lipschitz continuous third-order derivatives then by
Lemma~\ref{lem:exploit-point-cubic}, 
$f(x\pm\sqrt{\frac{2\alpha}{\SmCubic}}v)< f(x) - 
	\frac{\alpha^2}{4\SmCubic}$.

We can 
combine \callEIG{} with Algorithm~\ref{alg:MAIN} that finds a point with a 
small 
gradient as follows:
\begin{subequations}\label{combined:grad-eig}
\begin{flalign}
\hat{z}_{k} &\gets \callMAIN{f, z_{k}, \SmGrad, \epsilon, \alpha, \eta} 
\label{eq:guarded-call}\\
v_{k} &\gets \callEIG{f, \hat{z}_{k}, \SmGrad, \alpha, \delta'} \\
z_{k+1} &\gets \argmin_{x \in \{ \hat{z}_{k} + \eta v_k, \hat{z}_{k} - \eta v_k 
\}} f(x) 
\label{eq:eig-step}
\end{flalign}
\end{subequations}
As discussed above, under third order smoothness , $\eta 
=\sqrt{{2\alpha}/{\SmCubic}}$ 
guarantees that the step~\eqref{eq:eig-step} makes at least 
$\alpha^2/(4\SmCubic)$ function progress whenever $v_k^T \hess f( 
\hat{z}_{k}) v_k \le   -\alpha/2$. Therefore the above iteration can run at most 
$\Ot(\DeltaF\SmCubic/\alpha^2)$ times before $v_k^T \hess f( \hat{z}_{k}) 
v_k  \ge -\alpha/2$ is satisfied. Whenever $v_k^T \hess f( \hat{z}_{k}) v_k \ge 
-\alpha/2$, with 
probability $1-\delta'\cdot k$ we have the Hessian guarantee $\hess f( 
\hat{z}_{k}) \succeq -\alpha I$. 
Moreover, $\norm{\grad f( 
\hat{z}_{k})} \le \epsilon$ always holds. Thus, by setting $\alpha = 
\SmCubic^{1/3}\epsilon^{2/3}$ we obtain the required second order 
stationarity guarantee upon termination of the 
iterations~\eqref{combined:grad-eig}.

It remains to bound the computational cost of the method, with $\alpha = 
\SmCubic^{1/3}\epsilon^{2/3}$. The total number 
of Hessian-vector products 
required by \callEIG{} is,
\begin{equation*}
\Ot\left( \DeltaF\SmCubic/\alpha^2 \cdot 
\sqrt{\frac{\SmGrad}{\alpha}}\log d \right) = 
\Ot\left(\DeltaF\SmGrad^{1/2}\SmCubic^{1/6}\epsilon^{-5/3}\log d\right).
\end{equation*}
Moreover, 
it 
is readily seen from the proof of Theorem~\ref{thm:main-result-third} that 
every evaluation of~\eqref{eq:guarded-call} requires at most 
\begin{equation}\label{eq:progress-runtime-gurantee}
\Ot( (f(x_k) - 
f(x_{k+1}))\SmGrad^{1/2}\SmCubic^{1/6}\epsilon^{-5/3} + 
\SmGrad^{1/2}\SmCubic^{-1/6}\epsilon^{-1/3})
\end{equation}
 gradient and function 
evaluations. 
By telescoping the first term and multiplying the second by 
$\Ot(\DeltaF\SmCubic/\alpha^2)$, we  guarantee  $\norm{\grad f(x)} \le 
\epsilon$ and $\hess f( x) 
\succeq -\SmCubic^{1/3}\epsilon^{2/3} I$ in at most 
$\Ot(\DeltaF\SmGrad^{1/2}\SmCubic^{1/6}\epsilon^{-5/3}\log d)$ function, 
gradient and Hessian-vector product evaluations.

The argument above is the same as the one used to prove Theorem 4.3 
of~\cite{carmon2016accelerated}, but our improved guarantees under third 
order smoothness allows us get a better $\epsilon$ dependence for the 
complexity and lower bound on the Hessian in that regime. If instead we use 
the second 
order smoothness setting, we recover exactly the guarantees of 
\cite{carmon2016accelerated,agarwal2016finding}, namely $\norm{\grad f(x)} 
\le 
\epsilon$ and $\hess f( x) 
\succeq -\SmHess^{1/2}\epsilon^{1/2} I$ in at most 
$\Ot(\DeltaF\SmGrad^{1/2}\SmHess^{1/4}\epsilon^{-7/4}\log d)$ function, 
gradient and Hessian-vector product evaluations.

Finally, we remark that the above analysis would still apply if 
in~\eqref{eq:guarded-call} we replace \callMAIN{} with any method with a run 
-time guarantee of the 
form~\eqref{eq:progress-runtime-gurantee}. The resulting method will  
guarantee whatever the original method does, and also $\hess f(x) \succeq 
-\alpha I$. In particular, if the first method guarantees a small gradient, the 
combined method guarantees convergence to second-order stationary points. 

\section{Experiment details}\label{app:experiments}

\subsection{Implementation details}\label{app:implementation}

\paragraph{Semi-adaptive gradient steps}

Both gradient descent and AGD are based on gradients steps of the form  
\begin{equation}\label{eq:gradient-step}
y_{t+1} = x_t - \frac{1}{\SmGrad}\grad f (x_t).
\end{equation}
In practice $\SmGrad$ is often unknown and non-uniform, and therefore 
needs to be estimated adaptively. A common approach is backtracking line 
search, which we use for conjugate gradient. However,  
combining line search with AGD without invalidating its 
performance guarantees would involve non-trivial modification of the 
proposed method. Therefore, for the rest of the methods we keep an 
estimate of $\SmGrad$, and double it whenever the gradient steps fails to 
make sufficient progress. That is, whenever
\begin{equation*}
f\left(x_t -\frac{1}{\SmGrad}\grad f (x_t)\right) > f(x_t) - 
\frac{1}{2\SmGrad}\norm{\grad f(x_t)}^2
\end{equation*}
we set $\SmGrad \gets 2\SmGrad$ and try again. In all experiments we start 
with $\SmGrad = 1$, which underestimates the actual smoothness of $f$ by 
2-3 orders of magnitude. We call our scheme for setting $\SmGrad$ 
semi-adaptive, since we only increase $\SmGrad$, and therefore do not adapt 
to situations where the function becomes more smooth as optimization 
progresses. Thus, we avoid painstaking tuning of $\SmGrad$ while 
preserving the `fixed step-size' nature of our approach, as $\SmGrad$ is only 
doubled a small number of times. 

\paragraph{Algorithm~\ref{alg:MAIN}}

We implement \callMAIN{} with the following modifications, indented to make 
it more practical 
without substantially compromising its theoretical properties. 

\begin{enumerate}
	\item We use the semi-adaptive scheme described above to set 
	$\localsmgrad$. Specifically, whenever the gradient steps in 
	lines~\ref{line:agd-y} and~\ref{line:zt} of \callAGDUPG{} and \callCERT{} 
	respectively fail, we double $\localsmgrad$ until it succeeds, terminate 
	\callAGDUPG{} and multiply $\SmGrad$ by the same factor.
	
	\item We make the input parameters for \callAGDUPG{} dynamic. In 
	particular, we set $\epsilon' = \norm{ \grad f( p_{k-1}) } / 10$ and use 
	$\alpha = \StrConv = C_{1} \norm{ \grad f( p_{k-1}) }^{2/3}$, where $C_1$ 
	is a hyper-parameter. We use the same value of $\alpha$ to construct 
	$\fhat$. This makes our implementation independent on the final desired 
	accuracy $\epsilon$. 
	
	\item In \callCERT{} we also test whether
	\begin{equation*}
	\fhat(x_t) + \grad \fhat (x_t)^T (y_t - x_t) > \fhat(y_t).
	\end{equation*}
	Since this inequality is a clear convexity violation, we return $w_t = y_t$ 
	whenever it holds. We find that this substantially increases our method's 
	capability of detecting negative curvature; most of the non-convexity 
	detection in the first experiment is due to this check.
	
	\item Whenever \callCERT{} produces a point $w_t \neq \NULL$ (thereby 
	proving non-convexity and stopping \callAGDUPG{}), instead of finding a 
	single pair $(v,u)$ that violates strong convexity, we compute
	\begin{equation*}
	\alpha_{v,u} = 2\frac{ f(v)-f(u)-\grad f(v)^T (u-v)}{\norm{u-v}^2}
	\end{equation*}
	for the $2t$ points of the form $v=x_j$ and $u=y_j$ or $u=w_t$, with $0\le 
	j < t$, where here we use the original $f$ rather than $\fhat$ given to 
	\callAGDUPG{}. We discard all pairs with $\alpha_{v,u}<0$ (no evidence of 
	negative curvature), and select the 5 pairs with highest value of 
	$\alpha_{v,u}$. For each selected pair $v,u$, we exploit negative curvature 
	by testing all the points of the form $\{z\pm \eta \delta\}$ with $\delta = 
	(u-v)/\norm{u-v}$, $z\in \{v,u\}$ and $\eta$ in a grid of 10 points 
	log-uniformly spaced between $0.01\norm{u-v}$ and 
	$100(\norm{u}+\norm{v})$.

	\item In~\callFBT{} we compute $c_j$ and $q_j$ for every $j$ such that 
	$f(x_j) > f(y_j)$. Moreover, when $v,u = \NULL$ (no non-convexity 
	detected), we still set the next iterate $p_k$ to be the output of $\callFBT{}$ 
	rather than just the last AGD step.
\end{enumerate} 

The hyper-parameter $C_1$ was tuned separately for each 
experiment by searching on a small grid. For the regression 
experiment the tuning was performed on different problem instances 
(different seeds) than the 
ones reported in Fig.~\ref{fig:tukey-problem}. 
For the neural network training problem the tuning was performed on a 
subsample of 10\%
one reported in Fig.~\ref{fig:ANN}. The specific parameters used were 
$C_1= 0.01$ for regression and $C_1 = 0.1$ for 
neural network training. 

%
%
%
%
%
%
%
%
%
%
%
%

\paragraph{Algorithm~\ref{alg:MAIN} without negative curvature 
exploitation} 

This method is identical to the one described above, except that at every 
iteration $p_k$ is set to $b^{(1)}$ produced by \callFBT{} (\ie the output of 
negative curvature exploitation is never used). We used the same 
hyper-parameters described above. 

\paragraph{Gradient descent}

Gradient descent descent is simply~\eqref{eq:gradient-step}, with $y_{t+1} = 
x_{t+1}$, where the semi-adaptive scheme is used to set 
$\SmGrad$.

\paragraph{Adaptive restart accelerated gradient descent}

We use the accelerated gradient descent scheme of 
\citet{beck2009gradient} with $\omega_{t} = t / (t + 3)$. We use the restart 
scheme given by \citet{o2015adaptive} where if $f(y_{t}) > f(y_{t-1})$ then we 
restart the algorithm from the point $y_{t}$. For the gradient steps we use the 
same semi-adaptive procedure described above and also restart the algorithm 
whenever the $\SmGrad$ estimate changes (restarts performed for this 
reason are not shown in Fig.~\ref{fig:tukey-problem} and~\ref{fig:ANN}).

\paragraph{Non-linear conjugate gradient} 

The method is given by the following recursion \cite{polaknote},
\begin{gather*}
\delta_t = -\grad f(x_t) + \max\left\{\frac{\grad f(x_{t})^T ( \grad f(x_{t}) 
-  \grad f(x_{t-1}) ) }{\| 
\grad f(x_{t-1}) \|^2} ,0\right\} \delta_{t-1} ~~,~~
x_{t+1} = x_{t} + \eta_t  \delta_t
\end{gather*}
where $\delta_0 = 0$ and $\eta_t$ is found via backtracking line search, as 
follows. If $\delta^T 
\grad f(x_t) \ge 0$ we set $\delta_t = -\grad 
f(x_{t})$ (truncating the recursion). We set 
$\eta_t = 2\eta_{t-1}$ and then check whether
\begin{equation*}%
f (x_{t} + \eta_t \delta_t )  \le f(x_t) + \frac{ \eta_t \delta_t^T \grad f(x_t) }{2}
\end{equation*}
holds. If it does we keep the value of $\eta_t$, and if it does not we set $\eta_t 
= \eta_t/2$ and repeat. The key difference from the semi-adaptive scheme 
used for the rest of the methods is the initialization $\eta_t = 2\eta_{t-1}$, that 
allows the step size to grow. Performing line search is crucial for conjugate 
gradient to succeed, as otherwise it cannot produce approximately conjugate 
directions. If instead we use the semi-adaptive step size scheme, 
performance becomes very similar to that of gradient descent.

\paragraph{Comparison of computational cost}

In the figures, the x-axis is set to the number of steps performed by 
the methods. We do this because it enables a one-to-one comparison 
between the steps of the restarted AGD and Algorithm~\ref{alg:MAIN}. 
However, Algorithm~\ref{alg:MAIN} requires twice the number of gradient 
evaluations per step of the other algorithms. Furthermore, the number of function evaluations of 
Algorithm~\ref{alg:MAIN}  increases substantially when we exploit negative 
curvature, due to our naive grid search procedure. 
Nonetheless, we believe it is possible to derive a variation of our approach 
that 
performs only one gradient computation per step, and yet maintains similar 
performance (see remark after Corollary~\ref{coro:witness-testimony}, and that effective negative curvature exploitation can be carried 
out with only few function evaluations, using a line search.  

While the rest of the methods tested require one gradient evaluation per step, the required number of function evaluations differs. GD requires only one function evaluation per step, while RAGD evaluates $f$ twice per step (at $x_t$ and $y_t$); the number of additional function evaluations due to the semi-adaptive scheme is negligible. NCG is expected to require more function evaluations due to its use of a backtracking line search. In the first experiment, NCG required 2 function evaluations per step on average, indicating that its $\SmGrad$ estimate was stable for long durations. Alg.~\ref{alg:MAIN} required 5.3 function evaluations per step (on average over the 1,000 problem instances, with standard deviation 0.5), putting the amortized cost of our crude negative curvature exploitation scheme at 3.3 function evaluations per step.

\subsection{Neural network training}\label{app:exp-2}

The function $f$ is the average cross-entropy loss of 10-way prediction of 
class labels from input features. The prediction if formed by applying softmax 
on the 
output of a neural network with three hidden layers of 20, 10 and 5 units and 
$\tanh$ activations. To obtain data features we perform the following 
preprocessing, where the training examples are treated as $28^2$ 
dimensional vectors. First, each example is separately normalized to zero 
mean and unit variance. Then, the $28^2\times 28^2$ data covariance matrix 
is formed, and a projection to the 10 principle components is found via 
eigen-decomposition. The projection is then applied to the training set, and 
then each of the 10 resulting features is normalized to have zero mean and 
unit variance across the training set. The resulting model has $d=545$ 
parameters and underfits the 60,000 examples training set. We randomly 
initialize the 
weights 
according the well-known scaling proposed 
by~\citet{glorot2010understanding}. We repeated the experiment for 10 
different initializations of the weights, and all results were consistent with 
those reported in Fig.~\ref{fig:ANN}.

\end{document}